\documentclass[11pt]{amsart}
\usepackage{amsmath,amssymb,amsthm}
\usepackage[english]{babel}
\usepackage{array}
\usepackage{tabu}
\usepackage[autostyle]{csquotes}
 2

\usepackage[colorlinks=true,linkcolor=blue,citecolor=blue,pdfpagelabels=false]{hyperref}
\newtheorem{thm}{Theorem}[section]
\newtheorem{lem}[thm]{Lemma}
\newtheorem{prop}[thm]{Proposition}
\newtheorem{cor}[thm]{Corollary}
\newtheorem{defn}[thm]{Definition}
\newcommand{\thmref}[1]{Theorem~\ref{#1}}
\newcommand{\lemref}[1]{Lemma~\ref{#1}}
\newcommand{\rmkref}[1]{Remark~\ref{#1}}
\newcommand{\propref}[1]{Proposition~\ref{#1}}
\newcommand{\corref}[1]{Corollary~\ref{#1}}

\newtheorem{rmk}[thm]{Remark}

\begin{document}

\baselineskip=17pt

\title[Certain congruences]
{Congruences in Hermitian Jacobi and Hermitian modular forms}
\author{Jaban Meher and Sujeet Kumar Singh}

\address[Jaban Meher]{School of Mathematical Sciences, National Institute of Science Education and Research, Bhubaneswar, HBNI, P.O. Jatni, Khurda 752050, Odisha, India.}
\email{jaban@niser.ac.in}

\address[Sujeet Kumar Singh]{School of Mathematical Sciences, National Institute of Science Education and Research, Bhubaneswar, HBNI, P.O. Jatni, Khurda 752050, Odisha, India.}
\email{sujeet.singh@niser.ac.in}

\subjclass[2010]{11F33, 11F55, 11F50}
\date{\today}

\keywords{Hermitian modular forms, Hermitian Jacobi forms, $U(p)$ congruences, Ramanujan-type congruences}

\maketitle
\begin{abstract}
In this paper we first prove an isomorphism between certain spaces of Jacobi forms. Using this isomorphism, we study the mod $p$ theory of Hermitian Jacobi forms over $\mathbb{Q}(i)$.
We then apply the mod $p$ theory of Hermitian Jacobi forms to characterize $U(p)$ congruences and to study Ramanujan-type congruences for Hermitian Jacobi forms and Hermitian modular forms of degree $2$ over $\mathbb{Q}(i)$.
\end{abstract}

\section{Introduction}
The Fourier coefficients of modular forms are related to many objects in number theory. Therefore there have been a great amount of research on studying the arithmetic properties of Fourier coefficients of modular forms and in general of different automorphic functions. In particular, a lot of research is based on studying various congruence properties of Fourier coefficients of different automorphic functions. The theory of Serre \cite{serre} and Swinnerton-Dyer \cite{swinnerton} on modular forms modulo a prime $p$ has a great impact in studying the congruences of Fourier coefficients of modular forms.
There are two kinds of congruences namely, $U(p)$ congruences and Ramanujan-type congruences which have attracted many mathematicians due to their various applications in number theory. 
Both $U(p)$ congruences and Ramanujan-type congruences are applications of the theory of Serre and Swinnerton-Dyer.
$U(p)$ congruences involve Atkin's $U$-operator. On the other hand, Ramanujan-type congruences are certain  kinds of congruences which were first studied by Ramanujan for the partition function $p(n)$. $U(p)$ congruences for elliptic modular forms have been studied by Ahlgren and Ono \cite{ahlgren}, Elkies, Ono and Yang \cite{elkies} and Guerzhoy \cite{guerzhoy}. We refer to the book of Ono \cite{ono} for a good overview of the $U(p)$ congruences. Ramanujan-type congruences for elliptic modular forms have been studied by Cooper, Wage and Wang \cite{cooper}, Dewar \cite{dewar1, dewar1.1} and Sinick 
\cite{sinick}. To prove results on $U(p)$ congruences and Ramanujan-type congruences for elliptic modular forms, one needs to study elliptic modular forms modulo a prime $p$ and prove certain results on filtrations of elliptic modular forms. $U(p)$ congruences for Siegel modular forms of degree $2$ were studied by Choi, Choie and Richter \cite{choies}. To prove results on $U(p)$ congruences, they used the results of Nagaoka \cite{nagaoka3} on Siegel modular forms of degree $2$ mod $p$ and certain results of Richter \cite{olav1, olav2} on Jacobi forms mod $p$. In fact, they proved certain results on filtrations of Siegel modular forms of degree $2$ and using those results on filtrations they proved the result on $U(p)$ congruences for Siegel modular forms of degree $2$. 
Raum and Richter \cite{richter} have studied $U(p)$ congruences for Siegel modular forms of any degree. On the other hand,
Ramanujan-type congruences for Jacobi forms and Siegel modular forms of degree $2$ were studied by Dewar and Richter \cite{dewar2} using the theories of Jacobi forms mod $p$ and Siegel modular forms of degree $2$ mod $p$. In this paper we study $U(p)$ congruences and Ramanujan-type congruences for Hermitian Jacobi forms and Hermitian modular forms of degree $2$ over $\mathbb{Q}(i)$. To study these results, one needs to know the theories of Hermitian Jacobi forms modulo $p$ and Hermitian modular forms modulo $p$. The theory of Hermitian Jacobi forms mod $p$ has been studied by Richter and Senadheera \cite{sena2}. But they have studied only Hermitian Jacobi forms of index $1$. In the same paper, using their results on Hermitian Jacobi forms mod $p$, they have proved a result on $U(p)$ congruences for Hermitian Jacobi forms of index $1$. Therefore if one wants to study $U(p)$ congruences for Hermitian Jacobi forms of any integer index, one needs to study the theory of Hermitian Jacobi forms mod $p$ for any integer index. Thus we first establish various results on Hermitian Jacobi forms mod $p$ for any integer index.
Using these results, we characterize $U(p)$ congruences and study Ramanujan-type congruences 
for Hermitian Jacobi forms of any integer index. Next we study Hermitian modular forms of degree $2$. Using the results of Kikuta and Nagaoka 
\cite{nagaoka1, nagaoka2} on Hermitian modular forms of degree $2$ modulo $p$ and our results on Hermitian Jacobi forms mod $p$, we characterize $U(p)$ congruences and study Ramanujan-type congruences for certain Hermitian modular forms of degree $2$.

   The paper is organised as follows. In Section 2, we recall some basics on Hermitian Jacobi forms over $\mathbb{Q}(i)$ and obtain some relations between Hermitian Jacobi forms and Jacobi forms. We also prove an isomorphism between two different spaces of Jacobi forms. This isomorphism is very crucial in proving some important results in Section 3.
In Section 3, we discuss Hermitian Jacobi forms modulo a prime $p$ and prove certain results on filtrations which are main ingredients to prove the main results in Section 4. In Section 4, we prove results on $U(p)$ congruences and Ramanujan-type congruences for Hermitian Jacobi forms of arbitrary integer index. In Section 5, we illustrate some examples to explain $U(p)$ congruences and Ramanujan-type congruences for Hermitian Jacobi forms. In Section 6, we recall some basics and known results on Hermitian modular forms of degree $2$ over $\mathbb{Q}(i)$. In Section 7, we use some results proved in Section 3 to prove a result on filtrations of Hermitian modular forms of degree $2$ modulo $p$. This result is one of the main ingredients in the proofs of the main results in Section 8. In Section 8, we prove results on $U(p)$ congruences and Ramanujan-type congruences for certain Hermitian modular forms of degree $2$. In Section 9, we provide some examples to illustrate the results proved in Section 8.

\section{Hermitian Jacobi forms}
Let $\mathcal{O}:=\mathbb{Z}[i]$ be the ring of integers of $\mathbb{Q}(i)$ with inverse different
$\mathcal{O}^{\#}=\frac{i}{2}\mathcal{O}$,
 let $\mathcal{O}^{\times}:=\{1, -1, i, -i\}$ be the set of units in $\mathcal{O}$. The Hermitian Jacobi group over $\mathcal{O}$ is $\Gamma^J(\mathcal{O})=\Gamma(\mathcal{O})\ltimes\mathcal{O}^2$, where $\Gamma(\mathcal{O})=\{\epsilon M ~|~M \in SL_2(\mathbb{Z}), \epsilon \in \mathcal{O}^{\times}\}$ is the Hermitian modular group. For any $r\in \mathbb{Q}(i)$, the norm of $r$ is defined by
 $N(r):=r\overline{r}$. Throughout the paper we use $e(z)=e^{2\pi i z}$ and $M^t$ as the transpose of the matrix $M$. Let $\mathcal{H}$ be the complex upper half-plane.
\begin{defn}
A holomorphic function $\phi:\mathcal{H}\times \mathbb{C}^2\longrightarrow \mathbb{C}$  is a Hermitian Jacobi form for $\Gamma^J(\mathcal{O})$ of weight $k$, index $m$ and parity $\delta\in \{\ +, -\}$ if for each 
$M=
\begin{pmatrix}
a & b \\
c & d
\end{pmatrix}
\in SL_2(\mathbb{Z})
$, $\epsilon \in \mathcal{O}^{\times}$ and $\lambda, \mu \in \mathcal{O}$, we have 

\begin{equation}\label{hj1}
\phi \mid_{k, m, \delta}\epsilon M(\tau, z_1, z_2):=\sigma(\epsilon)\epsilon^{-k}(c\tau+d)^{-k}e^{\frac{-2\pi imcz_1z_2}{c\tau+d}}\phi\left(\frac{a\tau +b}{c\tau+d}, \frac{\epsilon z_1}{c\tau+d}, \frac{\overline{\epsilon} z_2}{c \tau+d}\right)=\phi(\tau, z_1, z_2),
\end{equation}
where  $\tau \in \mathcal{H}$, $z_1, z_2 \in \mathbb{C}$ and
\[ \sigma(\epsilon)=\begin{cases} 
      1 & ~~~~\text{if}~~~~\delta=+, \\
      \epsilon^2 &~~~~ \text{if}~~~~ \delta=-,
   \end{cases}
\]
\begin{equation}\label{hj2}
\phi\mid_{m}[\lambda, \mu](\tau, z_1, z_2):=e^{2\pi im (\lambda\overline{\lambda}\tau +\overline{\lambda}z_1+\lambda z_2)} \phi\left(\tau, z_1+\lambda \tau +\mu, z_2+\overline{\lambda}\tau +\overline{\mu} \right)=\phi(\tau, z_1, z_2), 
\end{equation}
and $\phi$ has a Fourier expansion of the form
\begin{equation}\label{hj3}
\phi(\tau, z_1, z_2)=\sum_{\substack{n \in \mathbb{Z}, r\in \mathcal{O}^{\#} \\  N(r)\le mn}} c(\phi; n, r)q^n\zeta_1^r\zeta_2^{\overline{r}},
\end{equation}
 where $q=e(z)$, $\zeta_1=e(z_1)$, $\zeta_2=e(z_2)$. We say that $\phi$ is a Hermitian Jacobi cusp form if in addition to the conditions \eqref{hj1}, \eqref{hj2} and \eqref{hj3}, $\phi$ also satisfies the condition that $c(\phi; n, r)=0$ whenever $mn=N(r)$ in the Fourier expansion given in \eqref{hj3}.
\end{defn}
We denote by  $HJ_{k, m}^{\delta}(\Gamma^J(\mathcal{O}))$ the finite dimensional vector space of all Hermitian Jacobi forms of weight $k$, index $m$ and parity $\delta$. 

\subsection{Jacobi forms and their relations with Hermitian Jacobi forms}
Consider the Jacobi group $\Gamma^1(\mathcal{O})=SL_2(\mathbb{Z})\ltimes \mathcal{O}^2$. A Jacobi form of weight $k$ and index $m$ on the group $\Gamma^1(\mathcal{O})$ satisfies the transformation properties \eqref{hj1} with $\epsilon=1$ and \eqref{hj2}, and it also has a Fourier expansion of the form given in \eqref{hj3}. We refer to \cite{das, sengupta} for more details on it. We denote by $J_{k, m}^1(\Gamma^1(\mathcal{O}))$ the vector space of  all Jacobi forms of weight $k$ and index $m$ on 
$\Gamma^1(\mathcal{O})$. We observe that
\begin{equation}\label{subset}
HJ_{k, m}^{\delta}(\Gamma^J(\mathcal{O})) \subset J_{k, m}^1(\Gamma^1(\mathcal{O})) \quad 
\text{for~each} ~~\delta \in \{+, -\}.
\end{equation}
Given $f\in J_{k, m}^1(\Gamma^1(\mathcal{O}))$, one constructs a Hermitian Jacobi form of weight $k$, index $m$ and parity $\delta$ by using the averaging operator
$$
A:J_{k, m}^1(\Gamma^1(\mathcal{O}))\rightarrow HJ_{k, m}^{\delta} (\Gamma^J(\mathcal{O}))
$$
defined by 
\begin{equation}\label{averaging}
f\mapsto\sum_{\epsilon \in \mathcal{O}^{\times}}f\mid_{k, m, \delta}\epsilon I,
\end{equation}
where $I$ is the identity matrix.

The theory of Jacobi forms was developed by 
Eichler and Zagier \cite{zagier} who systematically studied  Jacobi forms of integer index. Later, Ziegler \cite{ziegler} introduced Jacobi forms of matrix index. Let $M$ be a symmetric, positive definite, half-integral $l \times l$ matrix with integral diagonal entries. Let $\Gamma^l:=SL_2(\mathbb{Z})\ltimes (\mathbb{Z}^{l} \times \mathbb{Z}^l)$ and let $U[V]=V^{t}UV$ for matrices $U$, $V$ of appropriate sizes.
\begin{defn}
A holomorphic function $\phi: \mathcal{H}\times \mathbb{C}^l \longrightarrow \mathbb{C}$ is a Jacobi form of weight $k$ and index $M$ if for each
$
\begin{pmatrix}
a & b \\
c & d
\end{pmatrix}
\in SL_2(\mathbb{Z})
$
we have
\begin{equation}\label{mi1}
\phi\mid_{k, M}(\tau, z_1, \cdots, z_l):=(c\tau+d)^{-k}e^{-2\pi i\frac{cM[z]}{c\tau+d}}
\phi\left(\frac{a\tau +b}{c\tau+d}, \frac{z_1}{c\tau+d},\cdots, \frac{z_l}{c\tau+d} \right)=\phi(\tau, z_1, \cdots, z_l),
\end{equation}
where $\tau\in \mathcal{H}$, $z=(z_1, z_2, \cdots, z_l)^{t}\in \mathbb{C}^l$,
\begin{equation}\label{mi2}
\phi\mid_M(\tau, z_1, \cdots, z_l):=e^{2\pi i(\tau M[\lambda]+2\lambda^{t}Mz)}\phi(\tau, z_1+\lambda_1\tau +\mu_1, \cdots, z_l+\lambda_l\tau+\mu_l)=\phi(\tau, z_1, \cdots, z_l), 
\end{equation}
where $\lambda=(\lambda_1, \lambda_2, \cdots, \lambda_l)^{t},  \mu=(\mu_1, \mu_2, \cdots, \mu_l)^{t}\in \mathbb{C}^l$
and $\phi$ has a Fourier expansion of the form 
\begin{equation}\label{mi3}
\phi(\tau, z_1,\cdots, z_l)=\sum_{\substack{0\le n\in \mathbb{Z}, r\in \mathbb{Z}^l \\ 4{\rm det}(M)n-M^{\#}[r]\ge 0}}c(\phi; n, r)q^n \zeta^r,
\end{equation}
where $q=e(\tau)$, $\zeta^r=e^{2\pi i r^{t}z}$ and $M^{\#}$ is the adjugate of $M$.
\end{defn}
We denote by $J_{k, M}(\Gamma^l)$ the complex vector space of Jacobi forms of weight $k$, matrix index $M$ on $\Gamma^l$.
We now prove an isomorphism which is the main tool in the proof of \thmref{hjf} in Section 3.
\begin{thm}\label{iso}
For an integer $m\ge 1$, let  $B$ denote the matrix 
$
\begin{pmatrix}
 m & 0 \\ 
 0 & m 
 \end{pmatrix}
 $. 
 Then the space $J_{k, m}^1(\Gamma^1(\mathcal{O}))$ is isomorphic to the space $J_{k, B}(\Gamma^2)$ as a vector space over $\mathbb{C}$.
\end{thm}
\begin{proof} 
For $f(\tau, z_1, z_2)\in J_{k, m}^1(\Gamma^1(\mathcal{O}))$, define
$$
\hat{f}(\tau, z_1, z_2)=f(\tau, z_1+iz_2, z_1-iz_2).
$$
Using the transformation properties of $f$, one sees that $\hat{f}$ satisfies the transformation properties
\eqref{mi1}, \eqref{mi2}. Suppose that the Fourier expansion of $f$ is given by
$$
f(\tau, z_1, z_2)=\sum_{\substack{n \in \mathbb{Z}, r\in \mathcal{O}^{\#} \\  N(r)\le mn}} c(n, r)e(n\tau+rz_1+\overline{r}z_2).
$$
Then 
$$
\hat{f}(\tau, z_1, z_2)=f(\tau, z_1+iz_2, z_1-iz_2)
=\sum_{\substack{ n \in \mathbb{Z}, r\in \mathcal{O}^{\#}\\  N(r)\le mn}}c(n, r)e((z_1+iz_2)r+(z_1-iz_2)\overline{r}).
$$
Let $r=\frac{\alpha}{2}+i\frac{\beta}{2}$, where $\alpha , \beta \in \mathbb{Z}$. Then define
$s=(\alpha, -\beta)^t\in \mathbb{Z}^2$. The correspondence $r=\frac{\alpha}{2}+i\frac{\beta}{2} \mapsto
s=(\alpha, -\beta)^t$ from $\mathcal{O}^{\#}$ to $\mathbb{Z}^2$ is bijective. Therefore we have
$$
\hat{f}(\tau, z_1, z_2)=
\sum_{\substack{n \in \mathbb{Z}, r\in \mathcal{O}^{\#}\\ 4mn-|r|^2\ge 0}}c(n, r)e(n\tau+\alpha z_1-\beta z_2)
=\sum_{\substack{n \in \mathbb{Z}, s\in \mathbb{Z}^{2}\\4\text{det}(B)n-B[s]\ge 0}}
c(n, r)e(n\tau+az_1+bz_2).
$$
Thus $\hat{f}$ has a Fourier expansion of the form given in  \eqref{mi3}. Therefore the map
$$
i: J_{k, m}^1(\Gamma^1(\mathcal{O})) \mapsto J_{k, B}(\Gamma^2)
$$
defined by
$$
f(\tau, z_1, z_2) \mapsto f(\tau, z_1+iz_2, z_1-iz_2)
$$
is a well-defined linear map. Similarly one proves that the map
$$
j: J_{k, B}(\Gamma^2) \mapsto J_{k, m}^1(\Gamma^1(\mathcal{O}))
$$
defined by
$$
g(\tau, z_1, z_2)\mapsto g\left(\tau, \frac{z_1+z_2}{2}, \frac{z_1-z_2}{2i}\right)
$$
is a well-defined linear map. 
Now it can be easily checked that $j\circ i=I_1$ and $i\circ j=I_2$, where $I_1$ and $I_2$ are the identity maps on the spaces $J_{k, m}^1(\Gamma^1(\mathcal{O}))$ and $J_{k, B}(\Gamma^2)$ respectively. This proves the theorem.
\end{proof}

Let $M_k(SL_2(\mathbb{Z}))$ denote the vector space of all modular forms of weight $k$ on $SL_2(\mathbb{Z})$ and let $M_{*}(SL_2(\mathbb{Z}))=\bigoplus_k M_k(SL_2(\mathbb{Z}))$ be the graded ring of all modular forms on $SL_2(\mathbb{Z})$. Let 
$J_{*, m}^1(\Gamma^1(\mathcal{O}))=\bigoplus_k J_{k, m}^1(\Gamma^1(\mathcal{O}))$ and $J_{*, B}(\Gamma^2)=\bigoplus_k J_{k, B}(\Gamma^2)$. The spaces $J_{*, m}^1(\Gamma^1(\mathcal{O}))$ and  $J_{*, B}(\Gamma^2)$ are modules over $M_*(SL_2(\mathbb{Z}))$. For a ring $R\subseteq \mathbb{C}$, let $M_k(SL_2(\mathbb{Z}), R)$ denote the set of all modular forms of weight $k$ having all the Fourier coefficients in $R$ and let $M_{*}(SL_2(\mathbb{Z}), R)=\bigoplus_k M_k(SL_2(\mathbb{Z}), R)$. Let $HJ_{k, m}^{\delta}(\Gamma^J(\mathcal{O}), R)$ denote the set of all Hermitian Jacobi forms of weight $k$, index $m$ and parity $\delta$ having all the Fourier coefficients in $R$. Let $J_{k, m}^1(\Gamma^1(\mathcal{O}), R)$ denote the set of all Jacobi forms in $J_{k, m}^1(\Gamma^1(\mathcal{O}))$ having all the Fourier coefficients in $R$ and let $J_{*, m}^1(\Gamma^1(\mathcal{O}), R)=\bigoplus_k J_{k, m}^1(\Gamma^1(\mathcal{O}), R)$. Similarly let  $J_{k, B}(\Gamma^2, R)$ denote the set of all Jacobi forms in $J_{k, B}(\Gamma^2)$ having all the Fourier coefficients in $R$ and let 
$J_{*, B}(\Gamma^2, R)=\bigoplus J_{k, B}(\Gamma^2, R)$.
 Let $\mathbb{Z}_{(p)}$ be the localization of $\mathbb{Z}$ at the prime $p$. The ring 
 $\mathbb{Z}_{(p)}$ is called the ring of $p$-integral rationals.
With these notations we have two important and immediate consequences of \thmref{iso}.
 \begin{cor}
$J_{*, m}^1(\Gamma^1(\mathcal{O}))$ is isomorphic to $J_{*, B}(\Gamma^2)$ as modules over $M_{*}(SL_2(\mathbb{Z}))$.
\end{cor}
\begin{cor}\label{isocor}
$J_{k, m}^1(\Gamma^1(\mathcal{O}), \mathbb{Z}_{(p)})$ is isomorphic to $J_{k, B}(\Gamma^2, \mathbb{Z}_{(p)})$ as modules over $\mathbb{Z}_{(p)}$. Moreover, $J_{*, m}^1(\Gamma^1(\mathcal{O}), \mathbb{Z}_{(p)})$ is isomorphic to $J_{*, B}(\Gamma^2, \mathbb{Z}_{(p)})$ as modules over $M_*(SL_2(\mathbb{Z}), \mathbb{Z}_{(p)})$.
\end{cor}
Let $\phi \in HJ_{k, m}^\delta(\Gamma^J(\mathcal{O}))$. Suppose that the Fourier expansion of $\phi$ is given by
$$
\phi(\tau, z_1, z_2)=\sum_{\substack{n \in \mathbb{Z}, r\in \mathcal{O}^{\#} \\  N(r)\le mn}} c(\phi; n, r)q^n\zeta_1^r\zeta_2^{\overline{r}}.
$$
For $\rho\in \mathcal{O}$ and $z\in \mathbb{C}$, define
$$
\phi[\rho](\tau, z)=\phi(\tau, \rho z, \overline{\rho} z).
$$
Using the transformation properties and the Fourier expansion of $\phi$, we observe that
$\phi[\rho](\tau, z)\in J_{k, N(\rho)m}(\Gamma^1)$. Moreover, the Fourier expansion of 
$\phi[\rho]$ is given by
$$
\phi[\rho](\tau, z)=\sum_{\substack{n\in \mathbb{Z}, r\in \mathcal{O}^{\#} \\(2 \Re(\rho r))^2\le 4mn}}c(\phi; n, r)q^n \zeta^{2\Re(\rho r)}
=\sum_{\substack{n\in \mathbb{Z}, a\in \mathbb{Z} \\ a^2\le 4mn}}c(\phi[\rho]; n, a)q^n 
\zeta^a,
$$
where $\Re(\rho r)$ is the real part of $\rho r$, $\zeta= e(z)$ and 
\begin{equation}\label{summation}
c(\phi[\rho]; n, a)=\sum_{\substack{r \in \mathcal{O}^{\#}, N(r)\le mn \\ 2\Re(\rho r)=a}}c(\phi; n, r).
\end{equation}
Therefore if $\phi \in HJ^{\delta}_{k, m}(\Gamma^J(\mathcal{O}), \mathbb{Z}_{(p)})$, then $\phi[\rho]\in J_{k, N(\rho)m}(\Gamma^1, \mathbb{Z}_{(p)})$.
We next prove the following result which will be crucially used in the proof of \thmref{hermitiancase} in Section 3. This result is a generalization of a result of Raum and Richter \cite[Proposition 2.5]{richter} to the case of Hermitian Jacobi forms.
\begin{prop}\label{non-zero}
Let $\phi\in HJ^{\delta}_{k, m}(\Gamma^J(\mathcal{O}))$. If $0\le n_0 \in \mathbb{Z}$ is fixed, then there exists an element $\rho \in \mathcal{O}$ such that for all $n\le n_0$ and $r\in \mathcal{O}^{\#}$ with $N(r)\le mn$, we have
\begin{equation}\label{cphi}
c(\phi[\rho]; n, 2\Re(\rho r))=c(\phi; n, r).
\end{equation}
Moreover, if $(\phi_k)_k$ is a finite family of Hermitian Jacobi forms with $\phi_k\in HJ^{\delta_k}_{k, m}(\Gamma^J(\mathcal{O}), \mathbb{Z}_{(p)})$ and $\phi_k \not \equiv 0 \pmod p$ for all $k$, then there exists an element $\rho \in \mathcal{O}$ such that $\phi_k[\rho] \not \equiv 0 \pmod p$ for all $k$.
\end{prop}
\begin{proof}
Choose an integer $b$ such that
$$
b > \mbox{max} \Big\{ |a_i|~\mid r=\frac{a_1}{2}+\frac{a_2}{2}i \in \mathcal{O}^{\#},~ N(r)\le mn_0 \Big\}.
$$
Let $\rho=1+4bi$. Assume that $r_1, r_2\in \mathcal{O}^{\#}$ and $n>0$ is an integer such that
$n\le n_0$ and $N(r_i)\le mn$ for $i=1,2$.
We first prove that
$2\Re(\rho r_1 )=2\Re(\rho r_2)$ if and only if $r_1=r_2$. Then by \eqref{summation}, \eqref{cphi} follows. It is trivial to see that if $r_1=r_2$ then $2\Re(\rho r_1 )=2\Re(\rho r_2)$. Conversely assume that
$2\Re(\rho r_1 )=2\Re(\rho r_2)$. Let
$$
r_1=\frac{a_1}{2}+\frac{a_2}{2}i~~~\mbox{and}~~~r_2=\frac{b_1}{2}+\frac{b_2}{2}i,
$$
where $a_1, a_2, b_1, b_2$ are integers. Then the statement $2\Re(\rho r_1 )=2\Re(\rho r_2)$ implies $a_1-b_1=4b(a_2-b_2)$. Since $N(r_i)\le mn_0$ for $i=1,2$, we then obtain
$$
|a_2-b_2|=\frac{1}{4b}|a_1-b_1| \le \frac{1}{2}.
$$
Therefore we deduce that $r_1=r_2$. To prove the second assertion of the proposition, assume that
$\phi_k \not\equiv 0 \pmod{p}$ for all $k$. For each $k$, let $n_k$ be the smallest integer such that there exists $r_k\in \mathcal{O}^{\#}$ with $c(\phi_k; n_k, r_k) \not \equiv 0 \pmod p$. Choose an integer $n_0$ such that $n_0>\mbox{max}\{n_k\}$. Then by the first assertion of this proposition, there exists $\rho\in \mathcal{O}$ such that for all $n \le n_0$ and $r\in \mathcal{O}^{\#}$ satisfying $N(r)\le mn$ we have 
$$
c(\phi_k[\rho]; n, 2\Re(\rho r))=c(\phi_k; n, r)
$$
for each $k$. In particular,  we have $c(\phi_k[\rho]; n_k , 2\Re(\rho r_k) )\not \equiv 0 \pmod p$ for each $k$. Hence $\phi_k[\rho]\not \equiv 0 \pmod p$ for all $k$.
\end{proof}

\subsection{Heat operator}
For any holomorphic function $\phi : \mathcal{H}\times  \mathbb{C}^2 \longrightarrow \mathbb{C}$,
the heat operator 
$$
L_m:=-\frac{1}{\pi^2}\left(2\pi i m \frac{\partial}{\partial \tau}-\frac{\partial^2}{\partial z_1 \partial z_2} \right)
$$
acts on $\phi$.
The following lemma gives the actions of $L_m$ on the spaces $J_{k, m}^1(\Gamma^1(\mathcal{O}))$ and  $HJ_{k, m}^{\delta}(\Gamma^J(\mathcal{O}))$. For a proof of the lemma we refer to 
\cite[Lemma 5.1]{senathesis}.
\begin{lem}\label{heatop}
Let $\phi: \mathcal{H} \times \mathbb{C}^2 \longrightarrow \mathbb{C}$ be a holomorphic function. 
Define
\begin{equation}\label{heatoperator}
\hat{\phi}=L_m(\phi)-\frac{(k-1)m}{3}E_2\phi,
\end{equation}
where $E_2$ is the Eisenstein series of weight $k$ on $SL_2(\mathbb{Z})$.
Then
\begin{itemize}
\item if $\phi \in J_{k, m}^1(\Gamma^1(\mathcal{O}))$ then $\hat{\phi}\in J_{k+2, m}^1(\Gamma^1(\mathcal{O}))$;
\item if $\phi \in HJ_{k, m}^{\delta}(\Gamma^J(\mathcal{O}))$ then $\hat{\phi}\in HJ_{k+2, m}^{-\delta}(\Gamma^J(\mathcal{O}))$.
\end{itemize}
\end{lem}

\section{Hermitian Jacobi forms modulo $p$}
Throughout this paper we assume that $p\ge 5$ is a prime and $\mathbb{F}_p$ is the finite field with $p$ elements. Suppose that $\phi \in HJ^{\delta}_{k, m}(\Gamma^J(\mathcal{O}), \mathbb{Z}_{(p)})$ and its Fourier expansion is given by
$$
\phi(\tau, z_1, z_2)=\sum_{\substack{n \in \mathbb{Z}, r\in \mathcal{O}^{\#} \\  N(r)\le mn}} c(\phi; n, r)q^n\zeta_1^r\zeta_2^{\overline{r}}. 
$$ 
The reduction $\overline{\phi}$ of $\phi$ modulo a prime $p$ is defined by
$$
\overline{\phi}(\tau, z_1, z_2)=\sum_{\substack{n \in \mathbb{Z}, r\in \mathcal{O}^{\#} \\  N(r)\le mn}} \overline{c}(\phi; n, r)q^n\zeta_1^r\zeta_2^{\overline{r}},
$$
where $\overline{c}(\phi; n, r)$ is the reduction of $c(\phi; n, r)$ modulo $p\mathbb{Z}_{(p)}$ (also written as $c(\phi; n, r)$ modulo $p$). We define
$$
HJ^{\delta}_{k, m}(\Gamma^J(\mathcal{O}), \mathbb{F}_{p})=\{ \overline{\phi}\mid \phi \in HJ^{\delta}_{k, m}(\Gamma^J(\mathcal{O}), \mathbb{Z}_{(p)})\}.
$$
The filtration of $\phi$ modulo $p$ is defined by
$$
\Omega(\phi)=\text{inf}\{k \mid \overline{\phi}  \in HJ^{\delta}_{k, m}(\Gamma^J(\mathcal{O}), \mathbb{F}_p)~\mbox{for}~\mbox{some}~ \delta \}.
$$
Similarly we  define
$$
J_{k, B}(\Gamma^2, \mathbb{F}_p)=\{\overline{\phi}\mid\phi \in J_{k, B}(\Gamma^2, \mathbb{Z}_{(p)})\}
$$
and
$$
J_{k, m}^1(\Gamma^1(\mathcal{O}), \mathbb{F}_p)=\{\overline{\phi}\mid \phi \in 
J_{k, m}^1(\Gamma^1(\mathcal{O}), \mathbb{Z}_{(p)}) \}.
$$
For $\phi \in J_{k, m}^1(\Gamma^1(\mathcal{O}), \mathbb{Z}_{(p)})$ we define its  filtration modulo $p$ by
$$
\omega(\phi)=\text{inf}\{k\mid \overline{\phi} \in J_{k, m}^1(\Gamma^1(\mathcal{O}), \mathbb{F}_p)\}.
$$
The next result is an extension of a result of Sofer \cite{sofer} on Jacobi forms to Hermitian Jacobi forms.

\begin{thm}\label{hermitiancase}
Suppose that $\phi \in HJ^{\delta_k}_{k, m}(\Gamma^J(\mathcal{O}), \mathbb{Z}_{(p)})$ and $\psi \in HJ^{\delta_{k'}}_{k', m'}(\Gamma^J(\mathcal{O}), \mathbb{Z}_{(p)})$ such that $0\not\equiv \phi \equiv \psi \pmod p$. Then $m=m'$ and $k\equiv k' \pmod {(p-1)}$. Moreover, if $m$ is fixed and $(\phi_k)_k$ is a finite family of Hermitian Jacobi forms with 
$\phi_k\in HJ^{\delta_k}_{k, m}(\Gamma^J(\mathcal{O}), \mathbb{Z}_{(p)})$ and $\sum_{k}\phi_k \equiv 0 \pmod p$, then for each $a\in \mathbb{Z}/(p-1)\mathbb{Z}$ we have
$$
\sum_{k\in a+(p-1)\mathbb{Z}} \phi_k \equiv 0 \pmod p.
$$
\end{thm}
 \begin{proof}
 We use the idea of the proof of \cite[Lemma 2.1]{sofer} to prove that $m=m'$.
Suppose that $\lambda, \mu \in \mathcal{O}^{\#}$ with $\lambda \not = 0$. Replacing $z_1$ by $z_1+\lambda \tau +\mu$, $z_2$ by $z_2+\overline{\lambda}\tau +\overline{\mu}$ and using transformation property \eqref{hj2} of Hermitian Jacobi forms on the  congruence $\phi \equiv \psi \pmod p$, we have 
\begin{equation}\label{psi}
(q^{|\lambda|^2}\zeta_1^{\overline{\lambda}}\zeta_2^{\lambda})^{-m}\phi \equiv (q^{|\lambda|^2}\zeta_1^{\overline{\lambda}}\zeta_2^{\lambda})^{-m'}\psi \pmod p.
\end{equation}
Therefore we have 
$$
(q^{|\lambda|^2}\zeta_1^{\overline{\lambda}}\zeta_2^{\lambda})^{-m}\phi \equiv (q^{|\lambda|^2}\zeta_1^{\overline{\lambda}}\zeta_2^{\lambda})^{-m'}\phi \pmod p,
$$ 
for every $\lambda \in \mathcal{O}^{\#}$ and hence $m=m'$. We observe that the statement
$k\equiv k' \pmod {(p-1)}$ follows from the second assertion of the theorem. Therefore we need only prove the second assertion of the theorem. We follow the idea of Raum and Richter 
\cite[Proposition 2.6]{richter} to prove the second assertion.
Let $m$ be fixed and let $\phi_k \in HJ^{\delta_k}_{k, m}(\Gamma^J(\mathcal{O}), \mathbb{Z}_{(p)})$ be such that $\sum_{k}\phi_k \equiv 0 \pmod p$. 
Then for any $s \in \mathcal{O}$ we have $ \phi_k[s] \in J_{k, N(s)m }(\Gamma^1, \mathbb{Z}_{(p)})$ and
$$
\sum_{k}\phi_k[s] \equiv 0 \pmod p.
$$
Then by \cite[Proposition 2.6]{richter} we have 
\begin{equation}\label{s}
\sum_{k\in a+(p-1)\mathbb{Z}}\phi_k[s]\equiv 0 \pmod p.
\end{equation}
If $0\le n_0 \in \mathbb{Z}$ is fixed, then by \propref{non-zero}, there exists an $\rho \in \mathcal{O}$ such that for all $n\le n_0$ and $r\in \mathcal{O}^{\#}$ with 
$N(r)\le mn$, we have
$c(\phi[\rho]; n, 2\Re(\rho r))=c(\phi; n, r).$
Therefore by \eqref{s}, for arbitrary $n$ and $r$ with $r\in \mathcal{O}^{\#}$ and $N(r)\le mn$, we have
$$
\sum_{k\in a+(p-1)\mathbb{Z}}c(\phi_k; n, r)\equiv 0\pmod p
$$
and hence we have 
$$
\sum_{k\in a+(p-1)\mathbb{Z}}\phi_k \equiv 0 \pmod p.
$$
\end{proof}

\begin{rmk}\label{jr}
We observe that an analogous result as \thmref{hermitiancase} for Jacobi forms on 
$\Gamma^1(\mathcal{O})$ can be proved similarly. One may either prove in a similar way as \thmref{hermitiancase} or use the isomorphism of \thmref{iso} and \cite[Proposition 2.6]{richter} to prove an analogous result for Jacobi forms on $\Gamma^1(\mathcal{O})$. In particular, if $f\in J^1_{k, m}(\Gamma^1(\mathcal{O}),\mathbb{Z}_{(p)})$ and 
$g \in J^1_{k', m}(\Gamma^1(\mathcal{O}), \mathbb{Z}_{(p)})$ are such that 
$0 \not \equiv f \equiv g \pmod p$, then $k \equiv k' \pmod {(p-1)}$
\end{rmk}
 Our next result is a crucial ingredient in the proofs of certain results on congruences in Hermitian Jacobi forms. Tate's theory of theta cycle of  a modular form (see \cite[Section 7]{jochnowitz}) relies on a similar result due to Swinnerton-Dyer  \cite[Lemma 5]{swinnerton} in the case of modular forms. Richter \cite[Proposition 2]{ olav2} has generalized the above mentioned result of Swinnerton-Dyer to the case of classical Jacobi forms. In the next result, we prove an analogous result in the case of Hermitian Jacobi forms.
\begin{thm}\label{hjf}
If $\phi \in HJ^{\delta}_{k, m}(\Gamma^J(\mathcal{O}), \mathbb{Z}_{(p)})$, then there exists 
$\psi \in HJ^{\delta '}_{k', m}(\Gamma^J(\mathcal{O}),  \mathbb{Z}_{(p)})$ 
for some integer $k'$ and $\delta'\in \{+, -\}$ such that $\overline{L_m(\phi)}=\overline{\psi}$.
Moreover, if $\phi \not \equiv 0 \pmod p$, then  
$$
\Omega(L_m(\phi)) \le \Omega(\phi) + p+1,
$$
with equality if and only if $p\nmid (\Omega(\phi)-1)m$.
\end{thm}
The method of proof of Richter \cite[Proposition 2]{ olav2} in the case of Jacobi forms can not be adopted directly to prove \thmref{hjf}. The main reason for this is the lack of certain structure of the space of Hermitian Jacobi forms. In the case of Jacobi forms, we have some structure available which was crucially used in the proof of \cite[Proposition 2]{ olav2}. However, we use the isomorphism between certain spaces of Jacobi forms proved in the last section to prove \thmref{hjf}. The remaining part of this section is devoted to the proof of \thmref{hjf}. We first state the following two results which are particular cases of three results of Raum and Richter \cite[Theorem 2.8, Proposition 2.11, Theorem 2.14]{richter}.
To state these results, we denote by $B$ the $2\times 2$ matrix 
$
\begin{pmatrix}
 m & 0 \\ 
 0 & m 
 \end{pmatrix}
 $
for an integer $m\ge 1$. 
\begin{lem}\label{previous}
The space $J_{*, B}(\Gamma^2, \mathbb{Z}_{(p)})$ is  a free module over $M_*(SL_2(\mathbb{Z}), \mathbb{Z}_{(p)})$ of rank $4m^2$  and it has a basis $\{\phi_1, \phi_2, \cdots, \phi_{4m^2}\}$ such that $\phi_i \in J_{k_i, B}(\Gamma^2, \mathbb{Z})$ for some integer $k_i$ for $1\le i\le 4m^2$. 
\end{lem}
\begin{lem}\label{previous1}
Let $\phi_i$ be as in the previous lemma. If
$\phi=\sum_{i=1}^{4m^2}f_i\phi_i \in J_{k, B}(\Gamma^2, \mathbb{Z}_{(p)})$ with  
$f_i\in M_{k-k_i}(SL_2(\mathbb{Z}), \mathbb{Z}_{(p)})$ and 
$\psi=\sum_{i=1}^{4m^2}g_i\phi_i \in J_{k', B}(\Gamma^2, \mathbb{Z}_{(p)})$ with
$g_i\in M_{k'-k_i}(SL_2(\mathbb{Z}), \mathbb{Z}_{(p)})$ are such that 
$0\not\equiv\phi \equiv \psi \pmod p$, then $f_i\equiv g_i \pmod p$.
\end{lem}
Using the isomorphism stated in \corref{isocor} we get the following immediate consequence of \lemref{previous} and \lemref{previous1}.
\begin{cor}\label{present}
The space
$J_{*, m}^1(\Gamma^1(\mathcal{O}), \mathbb{Z}_{(p)})$ is a free module of rank $4m^2$ over $M_*(SL_2(\mathbb{Z}), \mathbb{Z}_{(p)})$. This space has a basis $\{\psi_1, \psi_2, \cdots, \psi_{4m^2}\}$ such that $\psi_i\in J^1_{k_i, m}(\Gamma^1(\mathcal{O}), \mathbb{Z})$ for some integer $k_i$ for $1\le i\le 4m^2$. Moreover, if $\phi=\sum_{i=1}^{4m^2}f_i\psi_i \in J^1_{k, m}(\Gamma^1(\mathcal{O}), \mathbb{Z}_{(p)})$ with $f_i\in M_{k-k_i}(SL_2(\mathbb{Z}), \mathbb{Z}_{(p)})$
and $\psi=\sum_{i=1}^{4m^2}g_i\psi_i \in J^1_{k', m}( \Gamma^1(\mathcal{O}), \mathbb{Z}_{(p)})$ 
with $g_i\in M_{k'-k_i}(SL_2(\mathbb{Z}), \mathbb{Z}_{(p)})$ are
such that $0\not\equiv\phi \equiv \psi \pmod p$, then $f_i\equiv g_i \pmod p$.
\end{cor}
Now we are ready to prove a result analogous to \thmref{hjf} for Jacobi forms on 
$\Gamma^1(\mathcal{O})$.
\begin{prop}\label{jump}
Let $p\ge 5$ be a prime.
If $\phi \in J_{k, m}^1(\Gamma^1(\mathcal{O}), \mathbb{Z}_{(p)})$, then there exists 
$\psi \in J^1_{k', m}(\Gamma^1(\mathcal{O}), \mathbb{Z}_{(p)})$ for some integer $k'$ such that
$\overline{L_m(\phi)}=\overline{\psi}$. Moreover, if $\phi \not \equiv 0 \pmod p$, then
$$
\omega(L_m(\phi)) \le \omega(\phi) + p+1,
$$
with equality if and only if $p\nmid (\omega(\phi)-1)m$.
\end{prop}
\begin{proof}
We broadly follow the idea of Richter \cite[Proposition 2]{olav2} to prove this proposition. Suppose that $w(\phi)=k$. It is well known that $E_{p-1}\equiv 1 \pmod p$ and $E_{p+1} \equiv E_2 \pmod p$, 
where $E_{p-1}$, $E_{p+1}$ and $E_2$ are the Eisenstein series on $SL_2(\mathbb{Z})$ of weights $p-1$, $p+1$ and $2$ respectively and $p\ge 5$. Therefore by \lemref{heatop} we have
$$
L_m(\phi)\equiv \hat{\phi}E_{p-1}+\frac{(k-1)m}{3}E_{p+1}\phi \pmod{p},
$$
and $\hat{\phi}E_{p-1}+\frac{(k-1)m}{3}E_{p+1}\phi \in 
J^1_{k+p+1, m}(\Gamma^1(\mathcal{O}), \mathbb{Z}_{(p)})$. This proves the first assertion of the proposition. Now let us assume that $\phi \not \equiv 0 \pmod p$.
Then from the above discussion we have $\omega(L_m(\phi))\le k+p+1$. If
 $p\mid (k-1)m$ then by \eqref{heatoperator} we obtain $\omega(L_m(\phi)) \le k+2 < k+p+1$. Conversely assume that $\omega(L_m(\phi)) < k+p+1$. Assume on the contrary that 
 $p\nmid (k-1)m$. Then by \eqref{heatoperator} we have $\omega\left(\frac{(k-1)m}{3}E_2 \phi \right) <k+p+1$. We shall prove that $\omega(E_2 \phi) =k+p+1$ which leads to a contradiction. By \corref{present} we can write $\phi=\sum_{i=1}^{4m^2} f_i \psi_i$, where 
 $\psi_i\in J^1_{k_i, m}(\Gamma^1(\mathcal{O}), \mathbb{Z})$ and
 $f_i\in M_{k-k_i}(SL_2(\mathbb{Z}), \mathbb{Z}_{(p)})$ for $1\le i\le 4m^2$. Since $w(\phi)=k$, there exists $i$ such that $w(f_i\phi_i)=k$. Also by \cite[Theorem 2, Lemma 5]{swinnerton}, $f_iE_2$ has the maximal filtration and therefore we find that $\omega (\phi E_2)=k+p+1$.
\end{proof}

If $f\in HJ^{\delta}_{k, m}(\Gamma^J(\mathcal{O}), \mathbb{Z}_{(p)})$, then since 
$HJ^{\delta}_{k, m}(\Gamma^J(\mathcal{O}), \mathbb{Z}_{(p)}) \subset J^1_{k, m}(\Gamma^1(\mathcal{O}), \mathbb{Z}_{(p)})$, 
both $\Omega(f)$ and $\omega(f)$
are defined. The following proposition shows that in fact, both are same.
\begin{prop}\label{relation}
Let $p\ge 5$ be a prime.
If $f\in HJ^{\delta}_{k, m}(\Gamma^J(\mathcal{O}), \mathbb{Z}_{(p)})$, then 
$
\Omega(f)=\omega(f).
$
\end{prop}
\begin{proof}
Since $HJ^{\delta}_{k, m}(\Gamma^J(\mathcal{O}), \mathbb{Z}_{(p)}) \subset J^1_{k, m}(\Gamma^1(\mathcal{O}), \mathbb{Z}_{(p)})$, we always have 
$$
\omega(f) \le \Omega(f).
$$
Suppose that $w(f)=l$. To prove $\omega(f) = \Omega(f)$, it is sufficient to prove that there exists a Hermitian Jacobi form $h\in HJ^{\delta'}_{l, m}(\Gamma^J(\mathcal{O}), \mathbb{Z}_{(p)})$ for some 
$\delta' \in \{+, -\}$ such that $f\equiv h \pmod{p}$. Since $w(f)=l$, there exists a Jacobi form 
$g\in J_{l, m }^1(\Gamma^1(\mathcal{O}), \mathbb{Z}_{(p)})$ such that 
\begin{equation}\label{imp}
f(\tau, z_1, z_2) \equiv g(\tau, z_1, z_2) \pmod p.
\end{equation}
By  \rmkref{jr}, we have $k-l=a(p-1)$ for some integer $a$.  Let $k-l \equiv 0 \pmod 4$ and $\epsilon \in \mathcal{O}^{\times}$. Replacing $z_1$ by $\epsilon z_1$ and $z_2$ by $\overline{\epsilon}z_2$, we deduce from \eqref{imp} that
$$
f(\tau, \epsilon z_1, \overline{\epsilon}z_2) \equiv g(\tau, \epsilon z_1, \overline{\epsilon}z_2) \pmod{p}.
$$
Using the transformation property \eqref{hj1} for $f$ in the above congruence, we obtain
$$
f(\tau, z_1, z_2)\equiv  \sigma(\epsilon) \epsilon^{-k}g(\tau, \epsilon z_1, \overline{\epsilon}z_2)  
\pmod p,
$$
which implies that
$$
f(\tau, z_1, z_2) \equiv g\mid_{l, m, \delta}\epsilon I \pmod p.
$$
Let us define 
$$
h(\tau, z_1, z_2)=\frac{1}{4}\sum_{\epsilon \in \mathcal{O}^{\times}} g\mid_{l, m , \delta}\epsilon I.
$$
Then from \eqref{averaging} we have 
$h(\tau, z_1, z_2) \in HJ_{l, m}^{\delta}(\Gamma^J(\mathcal{O}), \mathbb{Z}_{(p)})$. Also it is clear that $f(\tau, z_1, z_2) \equiv h(\tau, z_1, z_2) \pmod{p}$. This proves that 
$\Omega(f)=\omega(f)$ if $k-l\equiv 0\pmod{4}$. If $k-l\equiv 0\pmod{2}$, then 
$h(\tau, z_1, z_2) \in HJ_{l, m}^{-\delta}(\Gamma^J(\mathcal{O}), \mathbb{Z}_{(p)})$. Then
one proves similarly that 
$\Omega(f)=\omega(f)$.
\end{proof}
\noindent{\bf Proof of \thmref{hjf}:}
Let $\phi \in HJ_{k, m}^{\delta}(\Gamma^J(\mathcal{O}), \mathbb{Z}_{(p)})$. We shall first prove that
\begin{equation}\label{cas}
\overline{L_m(\phi)} \in \begin{cases} 
       HJ_{k+p+1}^{\delta}(\Gamma^J(\mathcal{O}), \mathbb{F}_p) & ~~~~\text{if}~~~~p\equiv 3 \pmod p, \\
       HJ_{k+p+1}^{-\delta}(\Gamma^J(\mathcal{O}), \mathbb{F}_p) &~~~~ \text{if}~~~~ p\equiv 1 \pmod p.
   \end{cases}
\end{equation}
By \lemref{heatop}, we have
$$
L_m(\phi)=\hat{\phi}+\frac{(k-1)m}{3}E_2\phi,
$$
where $\hat{\phi}\in HJ_{k+2, m}^{-\delta}(\Gamma^J(\mathcal{O}), \mathbb{Z}_{(p)})$. Since
$$
E_{p-1} \equiv 1{\pmod p}~~~\mbox{and}~~~E_{p+1} \equiv E_2{\pmod p},
$$
we have
$$
L_m(\phi)\equiv \hat{\phi}E_{p-1}+\frac{(k-1)m}{3}E_{p+1}\phi \pmod{p}.
$$
Let $g=\hat{\phi}E_{p-1}+\frac{(k-1)m}{3}E_{p+1}\phi \pmod{p}$. Then 
$g\in J^1_{k+p+1, m}(\Gamma^1(\mathcal{O}), \mathbb{Z}_{(p)})$. Let $p\equiv 3 \pmod{4}$.
We will prove that $g\in HJ^{\delta}_{k+p+1, m}(\Gamma^J(\mathcal{O}), \mathbb{Z}_{(p)})$ by doing a straightforward computation. To prove 
$g\in HJ^{\delta}_{k+p+1, m}(\Gamma^J(\mathcal{O}), \mathbb{Z}_{(p)})$, it is sufficient to prove that
$$
g\mid_{k+p+1, m, \delta} \epsilon I= g
$$
for any $\epsilon\in \mathcal{O}^\times$. To prove this one easily checks that 
$$
\hat{\phi}E_{p-1}\mid_{k+p+1, m, \delta} \epsilon I = \hat{\phi}E_{p-1}
~~~\mbox{and}~~~
E_{p+1} f \mid_{k+p+1, m, \delta} \epsilon I = E_{p+1}f.
$$ 
This proves \eqref{cas} for $p\equiv 3 \pmod{4}$. The case for $p\equiv 1 \pmod{4}$ is similarly done.
Now by \propref{relation}, we have 
$$
\Omega(\phi)=\omega(\phi) \quad \text{and} \quad \Omega(L_m(\phi))=\omega(L_m(\phi)), 
$$
Therefore by \propref{jump}, \thmref{hjf} follows.
\section{Congruences in Hermitian Jacobi forms}
Let $p\ge 5$ be a prime. 
Let $\phi$ be a formal series of the form
$$
\phi=\sum_{\substack{n \in \mathbb{Z}, r\in \mathcal{O}^{\#} }}c(\phi; n, r)q^n \zeta_1^r\zeta_2^{\overline{r}},
$$
where $c(\phi; n, r)\in \mathbb{Z}_{(p)}$.
The heat operator $L_m$ acts on $\phi$ by
$$
L_m(\phi)=\sum_{\substack{n \in \mathbb{Z}, r\in \mathcal{O}^{\#} }} 4(nm-N(r))c(\phi; n, r)q^n \zeta_1^r \zeta_2^{\overline{r}}.
$$
We call the finite sequence $L_m^1(\phi):=L_m(\phi), L_m^2(\phi),\cdots L_m^{p-1}(\phi)$, the heat cycle of $\phi$. We observe that $L_m^{j+p-1} (\phi)\equiv  L_m^j(\phi)\pmod p$ for any integer 
$j\ge 1$. We say that 
$\phi$ is in its own heat cycle if $L_m^{p-1}(\phi) \equiv \phi \pmod p$. Now assume that
$\phi \in HJ_{k, m}^{\delta}(\Gamma^J(\mathcal{O}), \mathbb{Z}_{(p)})$, $\phi \not \equiv 0 \pmod p$ and $p\nmid m$. If $\Omega(L_m^i(\phi))\equiv 1 \pmod p$ for some integer $i\ge 1$, then we call $L_m^i(\phi)$ a high point and $L_m^{i+1}(\phi)$ a low point of the heat cycle. Suppose that $L_m(\phi)\not\equiv 0 \pmod{p}$ and $L_m^i(\phi)$ is a high point in the heat cycle. Then by \thmref{hjf}, we have
$$
\Omega(L_m^{i+1}(\phi))<\Omega(L_m^i(\phi)) +p+1.
$$
Also by \propref{hermitiancase} we have
\begin{equation}\label{heatcycle}
\Omega(L_m^{i+1}(\phi))=\Omega(L_m^i(\phi)) +p+1-s(p-1)
\end{equation}
for some integer $s\ge 1$.
We first prove the following important lemma which will be used to prove results on $U(p)$ congruences and Ramanujan-type congruences in this section. 
\begin{lem}\label{heatprop}
Let $p\ge 5$ be a prime.
Let $\phi \in HJ^{\delta}_{k, m}(\Gamma^J(\mathcal{O}), \mathbb{Z}_{(p)})$ 
for some $\delta \in \{+, -\}$. 
Suppose that $p\nmid m$ and $L_m(\phi)\not \equiv 0 \pmod p$.
\begin{itemize}
\item If $j \ge 1$, then $\Omega(L_m^j(\phi)) \not \equiv 2 \pmod p$.
\item The heat cycle of $\phi$ has one low point if and only if there is some $j \ge 1$ with $\Omega(L_m^j(\phi))\equiv 3 \pmod p$. In this case the low point is $L_m^j(\phi)$.
\item For any $j \ge 1$, $\Omega(L_m^{j+1}(\phi))\not = \Omega (L_m^{j}(\phi))+2$.
\item The number of low points of the heat cycle of $\phi$ is either one or two.
\end{itemize}
\end{lem}
\begin{proof}
Suppose that $\Omega(L_m^j(\phi)) \equiv 2 \pmod p$. Then $p\nmid (\Omega(L_m^j(\phi))-1)m$. 
Using \thmref{hjf} inductively we obtain
$$
\Omega(L_m^{j+n}(\phi))=\Omega(L_m^{j}(\phi))+n(p+1)
$$
for any integer $n$ with $1\le n \le p-1$. Since $L_m^j(\phi)\equiv L_m^{j+p-1}(\phi) \pmod{p}$ for any $j\ge 1$, 
in particular for $n=p-1$, we have
$$\Omega(L_m^j(\phi))=\Omega(L_m^{j+p-1}(\phi))=\Omega(L_m^j(\phi))+(p-1)(p+1).$$
This gives a contradiction. This proves the first assertion.

Suppose that $\Omega(L_m^j(\phi)) \equiv 3 \pmod p$. Applying \thmref{hjf} inductively
we have 
\begin{equation}\label{1}
\Omega(L_m^{j+n}(\phi))=\Omega(L_m^j(\phi))+n(p+1)
\end{equation}
for $1 \le n \le p-2$. Since $\Omega(L_m^{j+p-2}(\phi))\equiv 1 \pmod{p}$, $L_m^{j+p-2}(\phi)$
 is a high point. Therefore by $\eqref{heatcycle}$, we obtain
$$
\Omega(L_m^j(\phi))=\Omega(L_m^{j+p-1}(\phi))=\Omega(L_m^j(\phi))+(p-1)(p+1)-s(p-1)
$$
for some integer $s\ge 1$. From the above identity we deduce that $s=p+1$ and $L_m^j(\phi)$ is a low point and from \eqref{1} we observe that this is the only low point. Conversely assume that there is only one low point in the heat cycle. Let $L_m^j(\phi)$ be the only low point. Then
$L_m^{j+p-2}(\phi)$ must be the high point and 
$$
\Omega(L_m^{j+n}(\phi))=\Omega(L_m^{j}(\phi))+n(p+1)
$$
for any integer $n$ with $1\le n\le p-2$. Since $\Omega(L_m^{j+p-2}(\phi)) \equiv 1\pmod{p}$, from the above identity we have $\Omega(L_m^{j}(\phi))\equiv 3 \pmod{p}$. This proves the second assertion.

Suppose that $\Omega(L_m^{j+1}(\phi))=\Omega(L_m^j(\phi))+2$, for some $j\ge 1$. Then by \thmref{hjf} we have 
$$
\Omega(L_m^j(\phi))\equiv 1 \pmod p
$$
Therefore $\Omega(L_m^{j+1}(\phi))\equiv 3 \pmod p$. Using \thmref{hjf} inductively we obtain
$$
\Omega(L_m^{j+1+n}(\phi))=\Omega(L_m^{j+1}(\phi))+n(p+1)
$$
for any any integer $n$ with $1\le n \le p-2$. In particular for $n=p-2$, we get
$$
\Omega(L_m^j(\phi))=\Omega(L_m^{j+1+p-2}(\phi))=\Omega(L_m^j(\phi))+2+(p-2)(p+1).
$$
This gives a contradiction, proving the third assertion.

The second assertion of this lemma gives the necessary and sufficient condition for a heat cycle to have only one low point. Now suppose that the number of high points in the heat cycle of $\phi$ is $t\ge 2$. For $1\le i_1 \le i_2 \le \cdots \le i_t\le p-1$, let $L_m^{i_j}(\phi)$ be the high points in the heat cycle of $\phi$. We assume that $i_{t+1}=i_1+(p-1)$ for our convenience. By $\eqref{heatcycle}$ and the third assertion of this Lemma, for each $j$ with $1\le j\le t$, there exists an integer $s\ge 2$ such that
\begin{equation}\label{fall}
\Omega(L_m^{i_{j}+1}(\phi))=\Omega(L_m^{i_j}(f))+(p+1)-s_j(p-1)\equiv 2+s_j \pmod p.
\end{equation}
Therefore we have
$$
\Omega(L_m(\phi))=\Omega(L_m^{1+p-1}(\phi))=\Omega(L_m(\phi))+(p-1)(p+1)-(p-1)
\sum_{j=1}^{t}s_j.
$$
From the above identity, we deduce that $\sum_{j=1}^{t}s_j=p+1$. Let $1\le j\le t-1$. From \eqref{fall}, we have
$$
\Omega(L_m^{i_{j+1}}(\phi))\equiv i_{j+1}-i_j +1+s_j \pmod{p}.
$$
Also since $L_m^{i_{j+1}}(\phi)$ is a high point, we have
$$
\Omega(L_m^{i_{j+1}}(\phi))\equiv 1 \pmod{p}.
$$ 
From the above two congruence relations, we have
$$
i_{j+1}-i_j +s_j \equiv 0\pmod{p}.
$$
Since $s_j\ge 2$, $0\le i_{j+1}-i_j\le p-1$ and  $\sum_{j=1}^{t}s_j=p+1$, we deduce that
$$i_{j+1}-i_j=p-s_j.$$
Now
$$
p-1=i_{t+1}-i_1=\sum_{j=1}^{t}(i_{j+1}-i_{j})=\sum_{j=1}^{t}(p-s_j)=tp-(p+1).
$$
From the above equality we deduce that $t=2$.
\end{proof}

\subsection{$U(p)$ congruences}
\begin{defn}
Let 
$$
\phi=\sum_{\substack{n \in \mathbb{Z}, r\in \mathcal{O}^{\#}}}c(\phi; n, r)q^n\zeta_1^r\zeta_2^{\overline{r}}
$$
be a formal series.
The Atkin's $U(p)$ operator on $\phi$ is defined by
$$
\phi\mid U(p)=\sum_{\substack{n \in \mathbb{Z}, r\in \mathcal{O}^{\#} \\ p\mid 4(mn-N(r))}}c(\phi; n, r)q^n \zeta_1^r\zeta_2^{\overline{r}}.
$$
\end{defn}
We observe that $\phi\mid U(p) \equiv 0 \pmod p$ if and only if $L_m^{p-1}(\phi) \equiv \phi \pmod p$ if and only if $c(\phi; n, r) \equiv 0 \pmod p ~\mbox{whenever}~4(nm-N(r)) \equiv 0\pmod p$. In the following theorem we give a characterization of $U(p)$ congruences for Hermitian Jacobi forms in terms of filtrations. The following result generalizes the result of Richter and Senadheera 
\cite[Theorem 1.2]{sena2} to Hermitian Jacobi forms of any integer index.
\begin{thm}\label{hjfzero}
Let $p\ge 5$ be a prime and let $k\ge 4$ be an integer.
Suppose $\phi \in HJ^{\delta}_{k, m}(\Gamma^J(\mathcal{O}), \mathbb{Z}_{(p)})$ is such that 
$\phi \not \equiv 0 \pmod p$ and $p\nmid m$. If $p > k$, then 
\[ \Omega(L_m^{p+2-k}(\phi))=\begin{cases} 
      2p+4-k & \mbox{if}~~\phi\mid U(p)\not \equiv 0 \pmod p, \\
      p+5-k &  \mbox{if}~~\phi\mid U(p)\equiv 0 \pmod p.
   \end{cases}
\]
\end{thm}
\begin{proof}
Suppose that $\phi\mid U(p) \equiv 0 \pmod p$. Therefore $L_{m}^{p-1}(\phi)\equiv \phi \pmod p$, i.e., $\phi$ is in its own heat cycle. Since $p>k$, $\phi$ is a low point of the heat cycle by \thmref{hjf}. Since $\Omega(\phi) \not\equiv 1 \pmod{p}$ as $p>k$, $\phi$ is not a high point, and therefore 
$\Omega(L_m(\phi)) >0$ by \thmref{hjf}. Thus $L_m(\phi)\not\equiv 0 \pmod{p}$. Therefore by \lemref{heatprop} heat cycle of $\phi$ has either one or two low points. If the heat cycle of $\phi$ has only one low point, then the low point is $\phi$ and $\Omega(\phi)\equiv 3\pmod{p}$. Then by 
 \thmref{hermitiancase}, $\Omega(\phi)=k-\alpha(p-1)$ for some integer $\alpha \ge 0$. Therefore the only possibility is that $\Omega(\phi)=k=3$. But by the hypothesis $k\not =3$. This implies that the heat cycle of $\phi$ has two low points. Since $L_m^{p-2}(\phi)$ is a high point,
 let $i_1$ be the integer  with $1\le i_1<p-2$ be such that $L_m^{i_1}(\phi)$ is the other high point.
 Since $\phi\not\equiv 0 \pmod{p}$ and $L_m(\phi)\not\equiv 0 \pmod{p}$, $\Omega(\phi)=k$. Therefore
 $$
\Omega(L_m^{i_1}(\phi))=k+i_1(p+1)\equiv k+i_1 \equiv 1 \pmod p.
$$
Thus the only possibility is that $i_1=p+1-k$. Let $s_1, s_2\ge 1$ be integers such that
$$
\Omega(L_m^{i_1+1}(\phi))=\Omega(L_m^{i_1}(\phi))+p+1-s_1(p-1),
$$
and
$$
\Omega(L_m^{p-2+1}(\phi))=\Omega(L_m^{p-2}(\phi))+p+1-s_2(p-1).
$$
We have proved in the fourth assertion of \lemref{heatprop} that $s_1+s_2=p+1$ and $p-2-i_1=p-s_1$. Thus we have $s_1=p-k+3$, $s_2=k-2$ and
$$
\Omega(L_m^{p+2-k}(\phi))=k+(p+2-k)(p+1)-(p-k+3)(p-1)=p+5-k.
$$
 
 Now assume that $\phi\mid U(p) \not \equiv 0 \pmod p$. Then by following an argument similar to the proof of \cite[Proposition 3]{olav1}, we deduce that $L_m(\phi)$ is a low point of the heat cycle. Therefore $L_m^{p-1}(\phi)$ is a high point. Suppose that $L_m^{p-1}(\phi)$ is the only high point of the heat cycle. Then by \thmref{hjf} we have
 $$
 \Omega(L_m(\phi))=k+p+1.
 $$
 Then $k+p+1\equiv 3 \pmod{p}$. This implies that $k\equiv 2 \pmod{p}$. Since $k<p$ and $k\ge 4$, this is not possible. Therefore the heat cycle has two low points. Let $1\le i<p-1$ be another high point of the heat cycle. Then since $\Omega(L_m(\phi))=k+p+1$, $L_m(\phi)\not\equiv 0\pmod{p}$. Let $s_1, s_2\ge 1$ be integers such that
 $$
\Omega(L_m^{i_1+1}(\phi))=\Omega(L_m^{i_1}(\phi))+p+1-s_1(p-1),
$$
and
$$
\Omega(L_m(\phi))=\Omega(L_m^{p-1+1}(\phi))=\Omega(L_m^{p-1}(\phi))+p+1-s_2(p-1).
$$
Also we have
$$
L_m^{i_1}(\phi)=L_m(\phi)+(i_1-1)(p+1)=k+p+1+(i_1-1)(p+1)\equiv k+i_1 \equiv 1 \pmod p.
$$
Then as done previously, we deduce that $i_1=p+1-k$ and $s_1=p-k+2$. Therefore we obtain
$$
\Omega(L_m^{p+2-k}(\phi))=p+k+1+(p+1-k)(p+1)-(p-k+2)(p-1)=2p+4-k.
$$
\end{proof}

\subsection{Ramanujan-type congruences}
\begin{defn}
Let $\phi=\sum_{n\in \mathbb{Z}, r \in \mathcal{O}^{\#}}c(\phi; n, r)q^n \zeta_1^r \zeta_2^{\overline{r}}$
be such that $c(\phi; n, r)\in \mathbb{Z}_{(p)}$. We say that $\phi$ has a Ramanujan-type congruence at $b\not\equiv 0 \pmod p$ if $c(\phi; n, r)\equiv 0 \pmod p$ whenever $4(nm- N(r)) \equiv b \pmod p$. 
\end{defn}
We observe that $\phi$ has a Ramanujan-type congruence at $b \pmod p$ if and only if $(q^{-\frac{b}{4m}}\phi)\mid U(p) \equiv 0 \pmod p$. It also can be seen that
$(q^{-\frac{b}{4m}}\phi)\mid U(p) \equiv 0 \pmod p$ if and only if
$L_m^{p-1}(q^{-\frac{b}{4m}}\phi) \equiv q^{-\frac{b}{4m}}\phi \pmod {p}$. Therefore 
$\phi$ has a Ramanujan-type congruence at $b \pmod{p}$ if and only if 
$L_m^{p-1}(q^{-\frac{b}{4m}}\phi) \equiv q^{-\frac{b}{4m}}\phi \pmod {p}$.
The main aim of this subsection is to prove \thmref{enequality}.
We first prove the following proposition which gives an equivalent condition on the existence of Ramanujan-type congruences for Hermitian Jacobi forms. A similar result for Jacobi forms has been proved by Dewar and Richter \cite[Proposition 2.4]{dewar2}.
\begin{prop}\label{notzero}
Let $\phi \in HJ^{\delta}_{k, m}(\Gamma^J(\mathcal{O}), \mathbb{Z}_{(p)})$. Then $\phi$ has a Ramanujan-type congruence at $b \pmod p$ if and only if $L_m^{\frac{p+1}{2}}(\phi) \equiv -\left(\frac{b}{p}\right)L_m(\phi) \pmod p$. 
\end{prop}
\begin{proof}
As in \cite[Proposition 2.4]{dewar2} if $b\not\equiv 0\pmod{p}$, then
$$
L_m^{p-1}(q^{-\frac{b}{4m}}\phi) \equiv 
q^{-\frac{b}{4m}}\sum_{i=0}^{p-1}b^{p-1-i}L_m^i(\phi)\pmod p.
$$
Therefore $\phi$ has a Ramanujan-type congruence at $b\not \equiv 0 \pmod p$ if and only if 
\begin{equation}\label{all}
\sum_{i=1}^{p-1}b^{p-1-i}L_{m}^i(\phi) \equiv 0 \pmod p.
\end{equation}
Since $\phi \in HJ^{\delta}_{k, m}({\Gamma^J(\mathcal{O}), \mathbb{Z}_{(p)}})\subset J_{k, m}^1(\Gamma^1(\mathcal{O}), \mathbb{Z}_{(p)})$, by \corref{present} we have 
$$
\phi =\sum_{j=1}^{4m^2}f_j \psi_j
$$
for $\psi_j \in J_{k_j, m}^1(\Gamma^J(\mathcal{O}), \mathbb{Z})$ and $f_j \in M_{k-k_j}(SL_2(\mathbb{Z}), \mathbb{Z}_{(p)})$. From the proof of \thmref{hjf}, we see that for any integer 
$i\ge 1$, there exists 
$\phi_i \in HJ^{\delta_i}_{k+i(p+1), m}(\Gamma^J(\mathcal{O}), \mathbb{Z}_{(p)})$ for some 
$\delta_i\in \{+, -\}$ such that
$L_m^i(\phi)\equiv \phi_i \pmod p$. Let 
$F_{i, j}\in M_{k+i(p+1)-k_j}(SL_2(\mathbb{Z}), \mathbb{Z}_{(p)})$ be such that
$$
\phi_i =\sum_{j=1}^{4m^2}F_{i,j} \psi_j.
$$
Then
\begin{equation}\label{L1}
L_m^i(\phi)\equiv \sum_{j=1}^{4m^2}F_{i, j}\psi_j \pmod p.
\end{equation}
Substituting this in \eqref{all} we deduce that $\phi$ has a Ramanujan-type congruence at $b\not \equiv 0 \pmod p$ if and only if 
$$
\sum_{j=1}^{4m^2}\left( \sum_{i=1}^{p-1}b^{p-1-i}F_{i, j}   \right)\psi_j \equiv 0 \pmod p.
$$
Therefore by \corref{present}, $\phi$ has a Ramanujan-type congruence at $b\not \equiv 0 \pmod p$ if and only if 
\begin{equation}\label{ramanujan-equi}
\sum_{i=1}^{p-1}b^{p-1-i}F_{i, j} \equiv 0 \pmod p
\end{equation}
By \cite[Theorem 2]{swinnerton},
\eqref{ramanujan-equi} is equivalent to
$$
b^{(p-1)/2-i}F_{i+(p-1)/2, j}+b^{p-1-i}F_{i, j}\equiv 0 \pmod p
$$
for all $1\le j\le 4m^2$ and $1\le i\le \frac{p-1}{2}$,
which is equivalent to the statement
\begin{equation}\label{ramanujan-equi2}
F_{i+(p-1)/2, j} \equiv - \left(\frac{b}{p}\right) F_{i, j} \pmod p
\end{equation}
for all $1\le j\le 4m^2$ and $1\le i\le \frac{p-1}{2}$. Therefore by \eqref{L1}, the above statement is equivalent to
\begin{equation}\label{ramanujan-equi3}
L_m^{i+\frac{p-1}{2}}(\phi)=\sum_{j=1}^{4m^2}F_{i+\frac{p-1}{2}, j}  \psi_j
\equiv \sum_{j=1}^{4m^2}-\left( \frac{b}{p} \right)F_{i, j}\psi_j
\equiv -\left( \frac{b}{p} \right)L_m^i (\phi) \pmod p
\end{equation}
for all $1\le i\le \frac{p-1}{2}$. Therefore in particular for $i=1$ we obtain
\begin{equation}\label{ramanujan-equi4}
L_m^{\frac{p+1}{2}}(\phi) \equiv -\left(\frac{b}{p}\right) L_m(\phi) \pmod p.
\end{equation}
Conversely if \eqref{ramanujan-equi4} holds, then by applying $L_m$ repeatedly on both sides of 
\eqref{ramanujan-equi4}, we obtain \eqref{ramanujan-equi3} for all $1\le i\le \frac{p-1}{2}$. This proves the proposition.
\end{proof}

As a consequence of the above proposition we have the following corollary.
\begin{cor}\label{heatenq}
Suppose that $\phi \in HJ^{\delta}_{k, m}(\Gamma^J(\mathcal{O}), \mathbb{Z}_{(p)})$ has a Ramanujan-type congruence at $b \pmod p$ and $L_{m}(\phi)\not \equiv 0 \pmod p$. Then the heat cycle of $\phi$ has two low points. Moreover, if $\Omega(\phi)=Ap+B$ with $1<B\le p-1$, then 
$$
\frac{p+3}{2} \le B \le A+\frac{p+3}{2}.
$$
\end{cor}
\begin{proof}
By the last proposition, $\phi$ has a Ramanujan-type congruence at $b\pmod{p}$ if and only if 
$L_m^{\frac{p+1}{2}}(\phi) \equiv -\left(\frac{b}{p}\right)L_m(\phi) \pmod p$. Therefore in this case we have $\Omega(L_m(\phi))=\Omega(L_m^{\frac{p+1}{2}}(\phi))=\Omega(L_m^{p}(\phi))$. Thus there must be one fall in the first half of the heat cycle and another fall in the second half of the heat cycle.
Therefore $\phi$ has two low points. Let $L_m^{i_1}(\phi)$ and $L_m^{i_2}(\phi)$ be the high points in the heat cycle, where $1\le i_1 \le \frac{p-1}{2}$ and $\frac{p+1}{2}\le i_2 \le p-1$. 
By \eqref{heatcycle} we have 
$$
\Omega(L_m^{i_1+1}(\phi))=\Omega(L_m^{i_1}(\phi))+p+1-s_1(p-1),
$$
and 
$$
\Omega(L_m^{i_2+1}(\phi))=\Omega(L_m^{i_2}(\phi))+p+1-s_2(p-1),
$$
for some $s_1, s_2 \ge 1$. Then by \propref{notzero} and \eqref{heatcycle}, we have 
$$
\Omega(L_m^{\frac{p+1}{2}}(\phi))=\Omega(L_m(\phi))+\frac{p-1}{2} (p+1)-s_1(p-1)=\Omega(L_m(\phi)),
$$
From the above identity we obtain $s_1=\frac{p+1}{2}$. Similarly one proves that 
$s_2=\frac{p+1}{2}$. Suppose now that $\Omega(\phi)=Ap+B$ with $1<B \le p-1$.
Since $L_m^{i_1}(\phi)$ is a high point, we have
$$
\Omega(L_m^{i_1}(\phi))= Ap+B+i_1(p+1) \equiv B+i_1 \equiv 1 \pmod p
$$
This implies that $B+i_1=p+1$ and $B\ge \frac{p+3}{2}$. Now the filtration of the low point $L_m^{p-B+2}(\phi)$ is given  by 
$$
\Omega(L_m^{p-B+2}(\phi))=Ap+B+(p-B+2)(p+1)-\frac{p+1}{2}(p-1).
$$
Since
$\Omega(L_m^{p-B+2}(\phi))\ge 0$, from the above identity, we obtain
$$
B\le A+\frac{p+3}{2}.
$$

\end{proof}

Our next result is the main result of this subsection.
\begin{thm}\label{enequality}
Let $\phi \in HJ^{\delta}_{k, m}(\Gamma^J(\mathcal{O}), \mathbb{Z}_{(p)})$ with 
$L_m(\phi) \not \equiv 0 \pmod p$. If $p> k$, $p\not= 2k-3$ and $p\nmid m$, then $\phi$ does not have a Ramanujan-type congruence at $b \pmod p$. \end{thm}
\begin{proof}
Assume that $\phi$ has a Ramanujan-type congruence at $b\pmod p$. First we observe that 
$\Omega(\phi)=k$. This is because the possible values of $\Omega(\phi)$ are $0$ or $k$. But since
$L_m(\phi) \not \equiv 0 \pmod p$, $\Omega(\phi)\not =0$. Now if $\Omega(\phi)=k=1$, then
by \thmref{hjf} and \thmref{hermitiancase}, $\Omega(L_m(\phi))=\Omega(\phi)+p+1-s(p-1)$ for some integer $s\ge 1$. Since $\Omega(L_m(\phi))\ge 0$, we have $s=1$. Then $\Omega(L_m(\phi))=3$.
Therefore by the third part of \lemref{heatprop}, we deduce that the heat cycle of $\phi$ has only one low point. This gives a contradiction to \corref{heatenq}. Thus $k\not =1$. Since $p>k$, if we write
$\Omega(\phi)=Ap+B$ as in \corref{heatenq}, then $A=0$ and $B=k$. Then by \corref{heatenq}, we obtain $p=2k-3$. This gives a contradiction to the hypothesis of the theorem.
\end{proof}

\section{Examples}

\subsection{$U(p)$ congruences}
 Let $f \in HJ_{k, m}^{\delta}(\Gamma^J(\mathcal{O}), \mathbb{Z}_{(p)})$. Suppose that for a given prime $p\ge 5$ we want to find out if $f\mid U(p)\equiv 0 \pmod{p}$. 
 If $k\ge 4$, $k<p$ and $p\nmid m$ we can apply \thmref{hjfzero}, otherwise  we need to check if $L_m^{p-1}(f) \equiv f \pmod p$.

We give examples of Hermitian Jacobi forms of index $1$. Some examples have been given by Richter and Senadheera \cite{sena2} and Senadheera \cite{senathesis}. We also explain how one gets more examples of Hermitian Jacobi forms of index $>1$ from Hermitian Jacobi forms of index $1$. 
Let $\phi_{4, 1}^+\in HJ_{4, 1}^+(\Gamma^J(\mathcal{O}), \mathbb{Z}_{(p)})$, $\phi_{6, 1}^-\in HJ_{6, 1}^-(\Gamma^J(\mathcal{O}), \mathbb{Z}_{(p)})$, $\phi_{8, 1}^+ \in HJ_{8, 1}^+(\Gamma^J(\mathcal{O}), \mathbb{Z}_{(p)})$ and $\phi_{10, 1}^{+} \in HJ_{10, 1}^{+, cusp}(\Gamma^J(\mathcal{O}), \mathbb{Z}_{(p)})$ be the Hermitian Jacobi forms defined in
\cite{senathesis}. For an even integer $k\ge 2$, let $E_k$ denote the Eisenstein series of weight $k$ on the full modular group $SL_2(\mathbb{Z})$.

Since $E_6 \equiv 1 \pmod 7$ and $E_2 \equiv E_4^2 \pmod 7$ we have 
$$
L_1^{6} (\phi_{10, 1}^{+, cusp}) \equiv (4E_2^5-E_2^3E_4+3E_2E_4^2+4E_2^2-2E_4)
\phi_{10, 1}^{+, cusp} \not\equiv \phi_{10, 1}^{+, cusp} \pmod 7. 
$$
Therefore $\phi_{10, 1}^{+, cusp}\mid U(7) \not \equiv 0\pmod{7}$. Also one checks that
$\phi_{10, 1}^{+, cusp}\mid U(11) \not \equiv 0\pmod{11}$ by \thmref{hjfzero}. 

For $\rho \in \mathcal{O}$, the index raising operator $\pi_{\rho}: HJ_{k, m}^{\delta}(\Gamma^J(\mathcal{O}), \mathbb{Z}_{(p)})\longrightarrow HJ_{k, N(\rho)m}^{\delta}(\Gamma^J(\mathcal{O}), \mathbb{Z}_{(p)})$ is defined by 
$$
f(\tau, z_1, z_2) \longmapsto f(\tau, \rho z_1, \overline{\rho} z_2).
$$
Therefore if $\rho \in \mathcal{O}$ be such that $p\nmid N(\rho)$,
$f\in HJ_{k, m}^{\delta}(\Gamma^J(\mathcal{O}), \mathbb{Z}_{(p)})$ and 
$f\mid U(p) \equiv 0 \pmod{p}$, then $\pi_{\rho}(f) \mid U(p) \equiv 0 \pmod{p}$.
We know from \cite{senathesis} that $\phi_{10, 1}^{+, cusp}\mid U(5) \equiv 0 \pmod{5}$. Therefore
$\pi_{(1+i)}(\phi_{10, 1}^{+, cusp})=\phi_{10, 1}^{+, cusp}(\tau, (1+i)z_1, (1-i)z_2) \in HJ_{10, 2}^+(\Gamma^J(\mathcal{O}), \mathbb{Z}_{(5)})$ and 
$\pi_{(1+i)}(\phi_{10, 1}^{+, cusp})\mid U(5) \equiv 0 \pmod{5}$.

\subsection{Ramanujan-type congruences}
We use the following two results to get examples of Hermitian Jacobi forms which have Ramanujan-type congruences. By \eqref{cas} and \propref{notzero} we obtain the following result.

\begin{thm}\label{sturm}
Let $\phi \in HJ_{k, m}^{\delta}(\Gamma^J(\mathcal{O}), \mathbb{Z}_{(p)})$ for some $\delta \in \{+, -\}$. If
$$
g:=L_m^{\frac{p+1}{2}}(\phi)+\left(\frac{b}{p}\right)L_m(\phi),
$$
then there exists $h \in HJ_{k+\frac{(p+1)^2}{2}}^{\delta}(\Gamma^J(\mathcal{O}), \mathbb{Z}_{(p)})$ such that $g\equiv h \pmod p$. Moreover, $\phi$ has a Ramanujan-type congruence at $b\not \equiv 0 \pmod p$ if and only if $g \equiv 0 \pmod p$.
\end{thm}

To apply \thmref{sturm}, we also require the following result. The result gives a Sturm bound for Hermitian Jacobi forms in characteristic $p$. Sturm bound for Hermitian Jacobi forms in characteristic $0$ has been obtained by Das \cite[Proposition 6.2]{das}. The proof of Das will go through in characteristic $p$ also. Therefore we do not give a proof of the following result. To state the result, define 
$$
\eta(k, m)=\bigg[\frac{4m^2(k-1)}{3} \prod_{p\mid 4m}\left(1-\frac{1}{p^2}\right)+\frac{m}{2}\bigg],
$$
where $p$ runs over all the prime divisors of $4m$.
\begin{prop}\label{sturm2}
Let $\phi \in HJ_{k, m}^{\delta}(\Gamma^J(\mathcal{O}), \mathbb{Z}_{(p)})$ for some $\delta\in \{+, -\}$ with Fourier expansion of the form \eqref{hj3}. If $c(\phi; n, r) \equiv 0 \pmod p$ for $0 \le n \le \eta(k, m)$, then $\phi \equiv 0 \pmod p$.
\end{prop}

To get some examples we apply \thmref{sturm}. To verify the congruence given in \thmref{sturm} we use \thmref{sturm2}. Therefore we need to check certain congruences for only finitely many coefficients. For these finitely many checking, we use SAGE. Also if 
$\phi \in HJ_{k, m}^{\delta}(\Gamma^J(\mathcal{O}), \mathbb{Z}_{(p)})$ and $p \nmid m$, by \thmref{enequality}, the only possibilities for Ramanujan-type congruences for $\phi$ are when $p \le k$ or $p=2k-3$. Following table gives some examples of Hermitian Jacobi forms having Ramanujan-type congruences.

\begin{center}
\begin{tabular}{ | c | c | c | c | c | c | } 
\hline
  Hermitian Jacobi forms & $b\pmod p$ \\
 \hline
 $\phi_{8, 1}^+$ & $b \equiv 1, 2, 4 \pmod 7 $  \\
\hline
$\phi_{8, 1}^+$ & $b \equiv 1, 3, 4, 9, 10, 12  \pmod {13} $  \\
\hline
$(E_6\phi_{4, 1}^+-E_4\phi_{6, 1}^-)/24$ &    $b \equiv 1, 2, 4 \pmod 7$     \\
\hline
 \end{tabular}
\end{center} 

\section{Hermitian modular forms}
The Hermitian upper half-space of degree $2$ is defined by 
$$
\mathcal{H}_2=\bigg\{Z=\begin{pmatrix}
\tau & z_1 \\
z_2 & \tau'
\end{pmatrix} \in M_2(\mathbb{C})\mid \frac{1}{2i}(Z-\overline{Z}^t )\ge 0  \bigg\},
$$
where $\overline{Z}^t$ is the transpose conjugate of the matrix $Z$. 
Let $J_2=\begin{pmatrix}
\bold{0} & I_2 \\
-I_2 & \bold{0}
\end{pmatrix}$, where $I_2$ denotes the $2 \times 2$ identity matrix and $\bold{0}$ denotes the $2\times 2$ zero matrix. Let  
$$
U_2:=\lbrace M\in M_{4}(\mathbb{C})\mid \overline{M}^t J_2M=J_2 \}. 
$$
The Hermitian modular group $\Gamma^2(\mathcal{O})$ of degree $2$ over $\mathbb{Q}(i)$ is defined by
$$
\Gamma^2(\mathcal{O})=M_4(\mathcal{O})\cap U_2.
$$

The group $\Gamma^2(\mathcal{O})$ acts on $\mathcal{H}_2$ by the fractional transformation 
$$
Z\longmapsto MZ= (AZ+B)(CZ+D)^{-1}, 
$$
where $M=\begin{pmatrix}
A & B \\
C & D
\end{pmatrix} \in \Gamma^2(\mathcal{O})$ and $Z\in \mathcal{H}_2$. Let $F$ be a complex valued function on $\mathcal{H}_2$. For a positive integer $k$  we define
$$
F\mid_k M(Z)=(\text{det}(CZ+D))^{-k}F(MZ), 
$$
where
$\text{det}$ is the determinant function and
$$
M=\begin{pmatrix}
A & B \\
C & D
\end{pmatrix} \in \Gamma^2(\mathcal{O}).
$$
For $k\in \mathbb{Z}$, let $\nu_k$ denote the abelian characters of $\Gamma^2(\mathcal{O})$ satisfying $\nu_k \cdot \nu_{k'}=\nu_{k+k'}$.
\begin{defn}
A holomorphic function $F:\mathcal{H}_2 \rightarrow \mathbb{C}$ is called a Hermitian modular form of weight $k$ and character $\nu_k$ on $\Gamma^2(\mathcal{O})$ if 
$$
F\mid_k M= \nu_k(M)F \quad \quad \text{for all} ~~  M \in \Gamma^2(\mathcal{O}).
$$
\end{defn}
Writing 
$Z=\begin{pmatrix}
\tau & z_1\\
z_2  & \tau'
\end{pmatrix}$, 
a Hermitian modular form $F$ has a Fourier expansion of the form
\begin{equation}\label{hmf}
F(Z)=\sum_{T \in {\Delta}_2} A_F(T)e(tr(TZ))=\sum_{\substack{n, m \in \mathbb{Z}, r\in \mathcal{O}^{\#} \\  N(r)\le mn}}A_F(n, r, m)q^n \zeta_1^r \zeta_2^{\overline{r}}(q')^{m},
\end{equation}
where 
$$
\Delta_2=
 \bigg\{ T=\begin{pmatrix}
n &   r \\
\overline{r} & m
\end{pmatrix}\ge 0 \mid  n, m \in \mathbb{Z},n \ge 0, m\ge 0, r \in \mathcal{O}^{\#}\bigg\},
$$
$tr(TZ)$ is the trace of the matrix $TZ$
and 
$q=e(\tau)$, $\zeta_1=e(z_1)$, $\zeta_2=e(z_2)$, $q'=e(\tau')$.

A Hermitian modular form $F$ is called a Hermitian cusp form if the sum in \eqref{hmf} runs over all positive-definite matrices $T\in \Delta_2$. We denote by $M_k(\Gamma^2(\mathcal{O}), \nu_k)$ the complex vector space of all Hermitian modular forms of weight $k$ and character $\nu_k$.
A Hermitian modular form $F\in M_k(\Gamma^2(\mathcal{O}), \nu_k)$ is called symmetric (respectively skew-symmetric) if 
$$
F(Z^{t})=F(Z) \quad \quad (\text{respectively}~ F(Z^{t})=-F(Z))
$$
for all $Z\in \mathcal{H}_2$.
We denote by $M_k(\Gamma^2(\mathcal{O}), \nu_k)^{sym}$ (respectively  $M_k(\Gamma^2(\mathcal{O}), \nu_k)^{skew}$) the subspace of 
$M_k(\Gamma^2(\mathcal{O}), \nu_k)$ consisting of
all symmetric (respectively skew-symmetric) Hermitian modular forms of weight $k$ and character $\nu_k$.
Writing 
$Z=\begin{pmatrix}
\tau & z_1 \\
z_2 & \tau'
\end{pmatrix}$, any $F\in M_k(\Gamma^2(\mathcal{O}), \nu_k)$ has a Fourier-Jacobi expansion
of the form:
\begin{equation}\label{decomposition}
F(Z)=F(\tau, z_1, z_2, \tau')=\sum_{m\ge 0}\phi_m(\tau, z_1, z_2)e(m\tau'),
\end{equation}
where $\phi_m \in HJ_{k, m}^{\delta}(\Gamma^J(\mathcal{O}))$ for some 
$\delta \in \{+, -\}$. We are interested in the case when $\nu_k=\text{det}^{k/2}$ ($k$ even), where  the character $\text{det}^{k/2}$ on $\Gamma^2(\mathcal{O})$ is defined by $M\mapsto \text{det}(M)^{k/2}$. Using a similar idea as in \cite[Theorem 7.1]{haver1}, we have the following result.

\begin{thm}\label{dparity}
Let $F\in M_k(\Gamma^2(\mathcal{O}), {\rm{det}}^{k/2})$. Suppose that the Fourier-Jacobi expansion 
of $F$ is given by
$$
F(\tau, z_1, z_2, \tau')=\sum_{m\ge 0}\phi_m(\tau, z_1, z_2)e(m\tau').
$$
Then $\phi_m$ is a Hermitian Jacobi form of weight $k$, index $m$
and parity $\delta$, where 
\[ \delta=\begin{cases} 
      + & \mbox{if}~~k\equiv 0 \pmod 4, \\
      - & \mbox{if}~~k\equiv 2 \pmod 4.
   \end{cases}
\]
\end{thm}
We define
$$
M(\Gamma^2(\mathcal{O}), \mbox{det}^{})^{sym}=\bigoplus_{k\in 2\mathbb{Z}}M_k(\Gamma^2(\mathcal{O}), \text{det}^{k/2})^{sym}.
$$
Then $M(\Gamma^2(\mathcal{O}), \mbox{det}^{})^{sym}$ is a graded ring. The Hermitian Eisenstein series of degree $2$ and even weight $k\ge 6$ is defined by
$$
H_k(Z)= \sum_{M}(\text{det}M)^{k/2}\text{det}(CZ+D)^{-k},
$$
where $M=
\begin{pmatrix}
* & *\\
C & D
\end{pmatrix}
$
runs over a set of representatives of $\bigg\{\begin{pmatrix}
* & * \\
0 & *
\end{pmatrix}  \bigg\}\setminus\Gamma^2(\mathcal{O})$. 
The Hermitian Eisenstein series $H_4$ of degree $2$ and weight $4$ has been constructed by the Maass lift in \cite{krieg}.
It is well-known that for even $k\ge 4$,
$$
H_k \in M_k(\Gamma^2(\mathcal{O}),                       \text{det}^{k/2})^{sym}.
$$
Using the Hermitian Eisenstein series, we obtain the symmetric Hermitian cusp forms
$$
\chi_8=-\frac{61}{230400}(H_8-H_4^2),
$$
$$
F_{10}=- \frac{277}{2419200}(H_{10}-H_4H_6),
$$
and 
$$
F_{12}=-\frac{34910011}{2002662144000}H_{12}-\frac{34801}{1009152000}H_4^3+\frac{414251}{9082368000}H_4H_8+\frac{50521}{8010648576}H_6^2
$$
of weights $8$, $10$ and $12$ respectively.
For any ring $R\subseteq \mathbb{C}$, we define 
$$
M_k(\Gamma^2(\mathcal{O}), \mbox{det}^{k/2}, R):=\bigg\{F=\sum_{T\in \Delta_2} A_F( T)e(tr(TZ))\in M_k(\Gamma^2(\mathcal{O}), \mbox{det}^{k/2} )\mid A_F(T)\in R \bigg\}
$$
and  
$$
M_k(\Gamma^2(\mathcal{O}), \mbox{det}^{k/2}, R)^{sym}:=\bigg\{F=\sum_{T\in \Delta_2} A_F( T)e(tr(TZ))\in M_k(\Gamma^2(\mathcal{O}), \mbox{det}^{k/2})^{sym} \mid A_F(T)\in R \bigg\}.
$$
Thus we have
$$
M(\Gamma^2(\mathcal{O}), \mbox{det}, R)^{sym}=\bigoplus_{k \in 2\mathbb{Z}} M_k(\Gamma^2(\mathcal{O}), \mbox{det}^{k/2}, R)^{sym}.
$$
We state the following result \cite[Theorem 4.3, Theorem 5.1]{nagaoka1}.
\begin{thm}
The symmetric Hermitian modular forms  $H_4$, $H_6$, $\chi_8$, $F_{10}$, $F_{12}$ are algebraically independent. 
If $F\in M_k(\Gamma^2(\mathcal{O}), {\rm{det}}^{k/2})^{sym}$, then there exists a polynomial $P_F \in \mathbb{C}[x_1, x_2, x_3, x_4, x_5]$ such that 
$$
F=P_F(H_4, H_6, \chi_8, F_{10}, F_{12}).
$$
In other words,
$$
\bigoplus_{k\in 2\mathbb{Z}}M_k(\Gamma^2(\mathcal{O}), {\rm{det}}^{k/2})^{sym}=\mathbb{C}[H_4, H_6, \chi_8, F_{10}, F_{12}].
$$
Moreover, the Hermitian modular forms $H_4$, $H_6$, $\chi_8$, $F_{10}$, $F_{12}$  have integral Fourier coefficients.
Furthermore, for any prime $p\ge 5$, if $F\in M_k(\Gamma^2(\mathcal{O}), {\rm{det}}^{k/2}, \mathbb{Z}_{(p)})^{sym}$, then there exists a polynomial $P \in \mathbb{Z}_{(p)}[x_1, x_2, x_3, x_4, x_5]$ such that 
$$
F=P(H_4, H_6, \chi_8, F_{10}, F_{12}).
$$
In other words,
$$
\bigoplus_{k\in 2\mathbb{Z}}M_k(\Gamma^2(\mathcal{O}), {\rm{det}}^{k/2}, \mathbb{Z}_{(p)})^{sym}=\mathbb{Z}_{(p)}[H_4, H_6, \chi_8, F_{10}, F_{12}].
$$
\end{thm}

\subsection{Heat operator}
The heat operator on 
any holomorphic function $F : \mathcal{H}_2 \longrightarrow \mathbb{C}$,
is defined by
$$
\mathbb{D}=-\frac{1}{\pi^2}\left(\frac{\partial^2}{\partial \tau \partial \tau'}-\frac{\partial^2}{\partial z_1 \partial z_2} \right).
$$
If $F\in M_k(\Gamma^2(\mathcal{O}), \mbox{det}^{k/2})$ has Fourier expansion of the form \eqref{hmf}, then the Fourier expansion of $\mathbb{D}(F)$ is given by
$$
\mathbb{D}(F)=\sum_{T\in \Delta_2}4\mbox{det}(T)A_F(T)e(tr(TZ))=\sum_{T=\begin{pmatrix}
n & r \\
\overline{r} & m
\end{pmatrix} \in \Delta_2}4(nm-N(r))A_F(n, r, m)q^n \zeta_1^r\zeta_2^{\overline{r}}(q')^m.
$$

\section{Hermitian modular forms modulo $p$}
In this section, $p\ge 5$ is a prime. Let $F\in M_k(\Gamma^2(\mathcal{O}), \mbox{det}^{k/2},  \mathbb{Z}_{(p)})^{sym}$ having Fourier expansion 
$$
F=\sum_{T \in \Delta_2} A_F(T)e(tr(TZ)).
$$
We define 
$$
\overline{F}=\sum_{T \in \Delta_2} \overline{A_F(T)}e(tr(TZ)),
$$
where $\overline{A_F(T)}$ is the reduction of $A_F(T)$ modulo $p$. Let
$$
M_k(\Gamma^2(\mathcal{O}), \mbox{det}^{k/2}, \mathbb{F}_p)=\{\overline{F}\mid F \in M_k(\Gamma^2(\mathcal{O}),\mbox{det}^{k/2}, \mathbb{Z}_{(p)})\},
$$

$$
M_k(\Gamma^2(\mathcal{O}), \mbox{det}^{k/2}, \mathbb{F}_p)^{sym}=\{\overline{F}\mid F \in M_k(\Gamma^2(\mathcal{O}),\mbox{det}^{k/2}, \mathbb{Z}_{(p)})^{sym} \}
$$
and
$$
M(\Gamma^2(\mathcal{O}), \mbox{det}^{k/2}, \mathbb{F}_p)^{sym}=\sum_{k\in 2\mathbb{Z}}{M}_k(\Gamma^2(\mathcal{O}), \mbox{det}^{k/2}, \mathbb{F}_p)^{sym}.
$$
For $F\in M_k(\Gamma^2(\mathcal{O}), \mbox{det}^{k/2}, \mathbb{Z}_{(p)})$ the filtration of $F$ modulo $p$ is defined by 
$$
\mho(F)=\mbox{inf}\big\{k \mid \overline{F}\in M_k(\Gamma^2(\mathcal{O}), \mbox{det}^{k/2}, \mathbb{F}_p)^{sym}\big\}.
$$
The main aim of this section is to prove \propref{hermitiancycle}.
For this we first prove a result similar to \thmref{hermitiancase} for Hermitian modular forms. A more general result for a symmetric Hermitian modular form has been proved by Kikuta \cite[Theorem 1.4]{kikuta}. But our method of proof is different and we prove it for any Hermitian modular form.
\begin{thm}\label{weightdifference}
Let $(F_k)_k$ be a finite family of Hermitian modular forms with $F_k \in M_k(\Gamma^2(\mathcal{O}), {\rm{det}}^{k/2}, \mathbb{Z}_{(p)})$. If $\sum_{k}F_k\equiv 0 \pmod p$, then for any $a\in \mathbb{Z}/(p-1)\mathbb{Z}$ we have
$$
\sum_{k\in a+(p-1)\mathbb{Z}}F_k \equiv 0 \pmod p.
$$ 
\end{thm}
\begin{proof}
Let $[F_k]_m$ denote the $m^{th}$ Hermitian Jacobi form in the Fourier-Jacobi expansion of $F_k$. Then by the Fourier-Jacobi expansion of $F_k$, we see that
$$
\sum_k F_k \equiv 0 \pmod p
$$ 
if and only if  
$$\sum_{k}[F_k]_m \equiv 0 \pmod p
$$ 
for all $m\ge 0$. By \thmref{hermitiancase} for each $a\in \mathbb{Z}/(p-1)\mathbb{Z}$ we have
$$
\sum_{k\in a+(p-1)\mathbb{Z}}[F_k]_m \equiv 0 \pmod p
$$
for all $m\ge 0$. This implies that
$$
\sum_{k\in a+(p-1)\mathbb{Z}}F_k \equiv 0 \pmod p.
$$
\end{proof}

Let $T=\sum_{}c_{(a, b, c, d, e)}x_1^ax_2^bx_3^cx_4^dx_5^e \in \mathbb{Z}_{(p)}[x_1, x_2, x_3, x_4, x_5]$ be a polynomial in the variables $x_1, x_2, x_3, x_4, x_5$. The reduction of $T$ modulo a prime $p$ is defined by
$$\overline{T}=\sum_{}\overline{c}_{(a, b, c, d, e)}x_1^ax_2^bx_3^cx_4^dx_5^e
\in \mathbb{F}_p[x_1, x_2, x_3, x_4, x_5],$$
where $\overline{c}_{(a, b, c, d, e)}$ is the reduction of $c_{(a, b, c, d, e)}$ modulo the prime $p$. With this definition, we recall the following result \cite[Proposition 5.1, Theorem 5.2]{nagaoka1}.
\begin{thm}\label{nagaokathm1}
Let $p\ge 5$  be a prime and let $F\in M_k(\Gamma^2(\mathcal{O}), {\rm{det}}^{k/2}, \mathbb{Z}_{(p)})^{sym}$.
Then there exists a Hermitian modular form $F_{p-1}\in M_{p-1}(\Gamma^2(\mathcal{O}), {\rm{det}}^{(p-1)/2}, \mathbb{Z}_{(p)})^{sym}$ such that 
$$
F_{p-1}\equiv 1 \pmod p.
$$
Furthermore, if $B \in \mathbb{Z}_{(p)}[x_1, x_2, x_3, x_4, x_5]$  is the polynomial defined by $F_{p-1}=B(H_4, H_6, \chi_8, F_{10}, F_{12})$, then the polynomial $\overline{B}-1$ is irreducible in 
$\mathbb{F}_p[x_1, x_2, x_3, x_4, x_5]$ and 
\begin{equation}
M(\Gamma^2(\mathcal{O}), {\rm{det}}^{k/2}, \mathbb{F}_p)^{sym}\cong \mathbb{F}_p[x_1, x_2, x_3, x_4, x_5]/(\overline{B}-1).
\end{equation}
\end{thm}
Using the above theorem we obtain the following important corollary. The proof of the corollary is similar to the proof of an analogous result in the elliptic modular form case \cite[Theorem 7.5 (i)]{lang}. Therefore we omit the proof of the corollary.
\begin{cor}\label{divide}
Let $F \in M_k(\Gamma^2(\mathcal{O}), {\rm{det}}^{k/2}, \mathbb{Z}_{(p)})^{sym}$ be such that
$F\not\equiv 0 \pmod{p}$. Suppose that $P_F\in \mathbb{Z}_{(p)}[x_1, x_2, x_3, x_4, x_5]$ is such that $F=P_F(H_4, H_6, \chi_8, F_{10}, F_{12})$. Then $\mho(F)< k$ if and only if $\overline{B}$ divides $\overline{P}_F$, where $B$ is as in \thmref{nagaokathm1}.
\end{cor}
Using the above corollary we obtain the following result which will be used in the proof of 
\propref{hermitiancycle}.
\begin{lem}\label{existencefiltration}
Let $p\ge 5$ be a prime. Suppose that 
$G\in M_{k}(\Gamma^2(\mathcal{O}), {\rm{det}}^{k/2}, \mathbb{Z}_{(p)})^{sym}$ is such that
$G\not\equiv 0 \pmod{p}$ and $\mho(G)=k$. Then there exist a positive integer $k'$ and a Hermitian modular form
$R\in M_{k'}(\Gamma^2(\mathcal{O}), {\rm{det}}^{k'/2}, \mathbb{Z}_{(p)})^{sym}$ with $R=P_R(H_4, H _6, \chi_8, F_{10}, F_{12})$ and $P_R\in \mathbb{Z}_{(p)}[x_1, x_2, x_3, x_4, x_5]$ such that 
$p \nmid k'(k'-1)$, $\mho(R)=k'$ and $\overline{B}$ does not divide the product $\overline{P}_R \overline{P}_G$, where $B$ is as in \thmref{nagaokathm1} and that $P_G\in \mathbb{Z}_{(p)}[x_1, x_2, x_3, x_4, x_5]$ is such that $G=P_G(H_4, H _6, \chi_8, F_{10}, F_{12})$. \end{lem}
\begin{proof}
Firstly consider the case when ${\rm{gcd}}(\overline{P}_G, \overline{B})=\overline{P}_{R}\not =1$. Let 
$$P_R(x_1, x_2, x_3, x_4, x_5)\in \mathbb{Z}_{(p)}[x_1, x_2, x_3, x_4, x_5]$$ 
be such that the reduction
of the polynomial $P_R(x_1, x_2, x_3, x_4, x_5)$ modulo $p$ is $\overline{P}_{R}$. Then it can be checked that $P_R(x_1, x_2, x_3, x_4, x_5)$ is a graded polynomial, i. e., 
$R:=P_R(H_4, H _6, \chi_8, F_{10}, F_{12})\in M_{k'}(\Gamma^2(\mathcal{O}), {\rm{det}}^{k'/2}, \mathbb{Z}_{(p)})^{sym}$ for some integer $k'>0$. Since $\mho(G)=k$, 
$\overline{P}_{R}\not = \overline{B}$ by \corref{divide}. 
Since $\overline{P}_{R}$ is a non-trivial factor of
$\overline{B}$, $k'<p-1$ and $\mho(R)=k'$ by \thmref{weightdifference}. Therefore $p \nmid k'(k'-1)$. Also since 
$\overline{P}_{R}\not = \overline{B}$, we observe that $\overline{B}$ does not divide 
$\overline{P}_R \overline{P}_G$. Next consider the case when 
${\rm{gcd}}(\overline{P}_G, \overline{B})=1$. Let $p>5$. From the Fourier expansion of $H_4$ it is clear that $H_4\not\equiv 0 \pmod{p}$. In fact, this is true for any prime $p$. Also since $p>5$, by \thmref{weightdifference} we have $\mho(H_4)=4$.
Thus if we consider $R=H_4$, then by \corref{divide}, $\overline{B}$ does not divide $\overline{P}_R$. Therefore 
$\overline{B}$ does not divide $\overline{P}_R \overline{P}_G$. Now suppose that $p=5$. It is clear from the Fourier expansion of $\chi_8$ that $\chi_8 \not\equiv 0 \pmod{5}$. Since $\chi_8$ is a cusp form, the possible values of $\mho({\chi_8})$ are $4$ and $8$. 
We need to prove that
$\mho({\chi_8})=8$. If $\mho({\chi_8})=4$, then 
$$
\chi_8 \equiv \alpha H_4 \pmod{5}
$$
for some $\alpha \in \mathbb{Z}_{(5)}$. The above congruence relation is not possible since the Fourier coefficient corresponding to the zero matrix of $H_4$ is $1$ where as that of $\chi_8$ is $0$. Therefore $\mho({\chi_8})=8$. Let us take $R=\chi_8$. Then from the above discussion and \corref{divide}, we deduce that 
$\overline{B}$ does not divide $\overline{P}_R$. Since ${\rm{gcd}}(\overline{P}_G, \overline{B})=1$, 
$\overline{B}$ does not divide $\overline{P}_R \overline{P}_G$.
\end{proof}
We next state the following result \cite[Theorem 3]{nagaoka2}.
\begin{thm}\label{lessthan}
Let $p\ge 5$ be a prime.
If $F\in M_k(\Gamma^2(\mathcal{O}), {\rm{det}}^{k/2}, \mathbb{Z}_{(p)})^{sym}$, then there is a cusp form $G\in M_{k+p+1}(\Gamma^2(\mathcal{O}), {\rm{det}}^{(k+p+1)/2}, \mathbb{Z}_{(p)})^{sym}$ such that 
$$
\mathbb{D}(F)\equiv G \pmod p.
$$
\end{thm}
We next recall Rankin-Cohen brackets of Hermitian modular forms which is a main ingredient in the proof of \propref{hermitiancycle}. Martin and Senadheera \cite{sena1} have defined Rankin-Cohen brackets of two Hermitian modular forms. We need only the first Rankin-Cohen bracket of two Hermitian modular forms for our purpose. Therefore we define only the first Rankin-Cohen bracket here. The first Rankin-Cohen bracket $[F_1, F_2]_1$ of two Hermitian modular forms $F_1$ and $F_2$ with 
$F_i\in M_{k_i}(\Gamma^2(\mathcal{O}), {\rm{det}}^{k_i/2})$ for $i=1,2$, is defined by
$$
[F_1, F_2]_1=(k_1-1)(k_2-1)\mathbb{D}(FG)-(k_2-1)(k_1+k_2-1)\mathbb{D}(F_1)F_2-(k_1-1)(k_1+k_2-1)F_1\mathbb{D}(F_2).
$$

We remark here that our definition of the first Rankin-Cohen bracket slightly different from the definition of Martin and Senadheera. But up to some constant multiple both the definitions are same. It is well known that with the above assumptions on $F_1$ and $F_2$, we have
$[F_1, F_2]_1\in  M_{k_1+k_2+2}(\Gamma^2(\mathcal{O}), {\rm{det}}^{(k_1+k_2+2)/2})$.
The following lemma follows from a straight forward computation.
\begin{lem}\label{p-integral}
If $F_1 \in M_{k_1}(\Gamma^2(\mathcal{O}), {\rm{det}}^{k_1/2} ,{\mathbb{Z}_{(p)}})^{sym}$ and $F_2 \in M_{k_2}(\Gamma^2(\mathcal{O}), {\rm{det}}^{k_2/2}, \mathbb{Z}_{(p)})^{sym}$,
then $[F_1, F_2]_1\in M_{k_1+k_2+2}(\Gamma^2(\mathcal{O}), {\rm{det}}^{(k_1+k_2+2)/2}, \mathbb{Z}_{(p)})^{sym}$.
\end{lem}
We next prove a result on filtrations which will be used to prove our main results of the next section.
\begin{prop}\label{hermitiancycle}
Let $F\in M_k(\Gamma^2(\mathcal{O}), {\rm{det}}^{k/2}, \mathbb{Z}_{(p)})^{sym}$. Suppose that there is an integer $m$ such that $p\nmid m$ and the $m^{th}$ Fourier-Jacobi coefficient $\phi_m$ of $F$ satisfies $\Omega(\phi_m)=\mho(F)$. Then 
$$
\mho(\mathbb{D}(F)) \le \mho(F)+p+1,
$$
with equality if and only if $p\nmid (\mho(F)-1)$.
\end{prop}
\begin{proof}
The proof is along a similar line of proof of \cite[Proposition 4]{choies}. If $\mho(F)=k'<k$, then there exists a Hermitian modular form 
$G\in M_{k'}(\Gamma^2(\mathcal{O}), \mbox{det}^{k'/2}, \mathbb{Z}_{(p)})$ such that 
$F\equiv G \pmod p$. Then we have $\mathbb{D}(F)\equiv \mathbb{D}(G)\pmod p$ and therefore we have $\mho(\mathbb{D}(F))=\mho(\mathbb{D}(G))$. Thus without loss of generality we assume that 
$\mho(F)=k$. Let 
$$
F(\tau, z_1, z_2, \tau')=\sum_{m=0}^{\infty} \phi_m(\tau, z_1, z_2) e(\tau')
$$
be the Fourier-Jacobi expansion of $F$. Then
$$
\mathbb{D}(F)=\sum_{m=0}^{\infty}L_m(\phi_m)e(\tau').
$$
By the hypothesis there is an integer $m$ such that $p\nmid m$ and $\Omega(\phi_m)=k$. If $p\nmid (k-1)$, then by  \thmref{hjf} one has $\Omega(L_m(\phi_m))=k+p+1$. Also for each non-negative integer $m$, we trivially observe that
$$
\Omega(L_m(\phi_m)) \le \mho(\mathbb{D}(F)).
$$
Also from \thmref{lessthan}, we have
$$
\mho(\mathbb{D}(F))\le k+p+1.
$$
Therefore we obtain
$$
\mho(\mathbb{D}(F))=k+p+1.
$$
Now conversely assume that $p\mid (k-1)$ and $\mho(\mathbb{D}(F))=k+p+1$. 
Since $\mho(\mathbb{D}(F))=k+p+1$, there exists $G\in M_{k+p+1}(\Gamma^2(\mathcal{O}), \mbox{det}^{(k+p+1)/2}, \mathbb{Z}_{(p)})^{sym}$ such that $\mathbb{D}(F)\equiv G \pmod p$. Let $P_G \in \mathbb{Z}_{(p)}[x_1, x_2, x_3, x_4, x_5]$ be such that 
$G=P_G(H_4, H_6, \chi_8, F_{10}, F_{12})$. Since $\mho(G)=k+p+1$, $G\not\equiv 0 \pmod{p}$.
Then by \lemref{existencefiltration}, there exists $R\in M_{k'}(\Gamma^2(\mathcal{O}), \mbox{det}^{k'/2}, \mathbb{Z}_{(p)})^{sym}$ with $P_R\in \mathbb{Z}_{(p)}[x_1, x_2, x_3, x_4, x_5] $ and $R=P_R(H_4, H _6, \chi_8, F_{10}, F_{12})$ such that $\mho(R)=k'$, $p \nmid k'(k'-1)$ and 
$\overline{B}$ does not divide the product $\overline{P}_R \overline{P}_G$. Therefore by \corref{divide} we have $\mho(GR)=k+k'+p+1$. 
Also by \lemref{p-integral} we have $[F, R]_1 \in M_{k+k'+2}(\Gamma^2(\mathcal{O}), \mbox{det}^{(k+k'+2)/2}, \mathbb{Z}_{(p)})^{sym}$ and 
$$
[F, R]_1 \equiv -(k'-1)k' \mathbb{D}(F)R\pmod p.
$$
Therefore
$$
k+k'+p+1=\mho(GR)=\mho(\mathbb{D}(F)R)=\mho([F, R]_1) \le k+k'+2.
$$
This gives a contradiction.

\end{proof}

\section{Congruences in Hermitian modular forms}
In this section we study $U(p)$ congruences and Ramanujan-type congruences for Hermitian modular forms.

\subsection{$U(p)$ congruences}
\begin{defn}
Let 
$$
F(\tau, z_1, z_2, \tau')=\sum_{\substack{n, m \in \mathbb{Z}, r\in \mathcal{O}^{\#}\\ nm-N(r)\ge 0}}A_F(n, r, m)q^n\zeta_1^r\zeta_2^{\overline{r}}(q')^{m} \in M_k(\Gamma^2(\mathcal{O}), {\rm{det}}^{k/2}).
$$
The Atkin's $U(p)$ operator on $F$ is defined by
$$
F\mid U(p)=\sum_{\substack{n, m \in \mathbb{Z}, r\in \mathcal{O}^{\#} \\ nm-N(r)\ge 0 \\p\mid 4(nm-N(r))}}A_F(n, r, m)q^n \zeta_1^r\zeta_2^{\overline{r}}(q')^m.
$$
\end{defn}
We have the following characterization of $U(p)$ congruences in terms of filtrations. This result generalizes the main result of Choi, Choie and Richter \cite[Theorem 1]{choies} to the case of Hermitian modular forms.
\begin{thm}\label{hermitiancong}
Let $p\ge 5$ be a prime.
Let 
$$
F(\tau, z_1, z_2, \tau')=\sum_{\substack{n, m \in \mathbb{Z}, r\in \mathcal{O}^{\#}\\ nm-N(r)\ge 0}}A_F(n, r, m)q^n\zeta_1^r\zeta_2^{\overline{r}}(q')^{m} \in M_k(\Gamma^2(\mathcal{O}), {\rm{det}}^{k/2}, \mathbb{Z}_{(p)})^{sym}
$$
with $p> k$. Assume that there exist $n, m\in \mathbb{Z}$ and $r\in \mathcal{O}^{\#}$ such that $p\nmid nm$ and $A_F(n, r, m)\not \equiv 0\pmod p$. Then we have
\[ \mho(\mathbb{D}^{p+2-k}(F)))=\begin{cases} 
      2p+4-k & {\rm{if}}~~ F\mid U(p)\not \equiv 0 \pmod p, \\
      p+5-k & {\rm{if}}~~F\mid U(p)\equiv 0 \pmod p.
   \end{cases}
\]
\end{thm}
\begin{proof}
Let 
$$
F(\tau, z_1, z_2, \tau')=\sum_{m \ge 0} \phi_{m}(\tau, z_1, z_2) e(m\tau')
$$
be the Fourier Jacobi expansion of $F$. We will first show that there exists an integer  
$m$ with $p\nmid m$ such that $\mho(F)=\Omega(\phi_m)$. Suppose on the contrary that for every integer $m$ with $p \nmid m$, we have $\Omega(\phi_m)<\mho(F)$. By the hypothesis 
$F\not\equiv  A_F(0, 0, 0)\pmod{p}$. Therefore since $p>k$, by \thmref{weightdifference} we have 
$\mho(F)=k$. Thus $\Omega(\phi_m)<k$ for each integer $m$ with $p\nmid m$. Therefore by 
\thmref{hermitiancase}, we have $\phi_m\equiv 0 \pmod{p}$ for each $m$ with $p\nmid m$, i.e.,
$A_F(n, r, m)\equiv 0 \pmod{p}$ for each $m$ with $p\nmid m$. Since $F(\tau, z_1, z_2, \tau')=F(\tau', z_1, z_2, \tau)$, we have $A_F(n, r, m)=A_F(m, r, n)$ and therefore we deduce that $A_F(n, r, m)\equiv 0 \pmod p$ for $p \nmid nm$. This gives a contradiction to the hypothesis of the theorem. Therefore there exists an integer $m$ with $p\nmid m$ such that 
$\Omega(\phi_m)=\mho(F)$. Now by using \thmref{hermitiancycle} and following a similar argument as in the proof of \thmref{hjfzero}, we get the required result.
\end{proof}

\subsection{Ramanujan-type congruences}
\begin{defn}
Let 
$$
F(\tau, z_1, z_2, \tau')=\sum_{\substack{n, m \in \mathbb{Z}, r\in \mathcal{O}^{\#}\\ nm-N(r)\ge 0}}A_F(n, r, m)q^n \zeta_1^r \zeta_2^{\overline{r}}(q')^m \in M_k(\Gamma^2(\mathcal{O}), {\rm{det}}^{k/2}, \mathbb{Z}_{(p)}).
$$ 
We say that $F$ has a Ramanujan-type congruence at $b \not \equiv 0\pmod p$ if $A_F(n, r, m)\equiv 0 \pmod p$ whenever $4(nm- N(r)) \equiv b \pmod p$. 
\end{defn}
In the next theorem, we prove results on existence and non-existence of Ramanujan-type congruences for symmetric Hermitian modular forms of degree $2$. A similar result for Siegel modular forms of degree $2$ has been proved by Dewar and Richter \cite[Theorem 1.2]{dewar2}. We follow their method of proof to prove our result.
\begin{thm}\label{nonzero}
Let $p\ge 5$ be a prime. Let 
$$
F(\tau, z_1, z_2, \tau')=\sum_{\substack{n, m \in \mathbb{Z}, r\in \mathcal{O}^{\#}\\  nm-N(r) \ge 0}}A_F(n, r, m)q^n \zeta_1^r \zeta_2^{\overline{r}}(q')^m \in M_k(\Gamma^2(\mathcal{O}), {\rm{det}}^{k/2}, \mathbb{Z}_{(p)})^{sym}.
$$
Then $F$ has a Ramanujan-type congruence at  $b \pmod p$ if and only if 
$$
\mathbb{D}^{\frac{p+1}{2}}(F)\equiv -\left(\frac{b}{p} \right) \mathbb{D}(F) \pmod p,
$$
where $\left(\frac{\cdot}{p} \right)$ is the Legendre symbol. Moreover, if $p> k$ with $p\not= 2k-3$ and there exist integers $n$ and $m$ such that $p\nmid nm$ and $A_F(n, r, m) \not \equiv 0 \pmod p$, then $F$ does not have a Ramanujan-type congruence at $b \pmod p$. 
\end{thm}
\begin{proof}
Let the Fourier-Jacobi expansion of $F$ be given by
$$
F(\tau, z_1, z_2, \tau')=\sum_{m=0}^{\infty}\phi_m(\tau, z_1, z_2) e(\tau').
$$
We observe that $F$ has a Ramanujan-type congruence at $b\pmod{p}$ if and only if $\phi_m$ has a Ramanujan-type congruence at $b\pmod{p}$ for all $m$. By \propref{notzero}, it is equivalent to the statement that for each $m$, we have
\begin{equation}\label{hermi}
L_m^{\frac{p+1}{2}}(\phi_m)\equiv -\left(\frac{b}{p}\right) L_m(\phi_m) \pmod p.
\end{equation}
Since 
$$
\mathbb{D}(F)=\sum_{m=0}^{\infty} L_m(\phi_m)e(\tau'),
$$
we deduce that $F$ has a Ramanujan-type congruence at  $b \pmod p$ if and only if 
$$
\mathbb{D}^{\frac{p+1}{2}}(F)=\sum_{m=0}^{\infty}L_m^{\frac{p+1}{2}}(\phi_m)e(\tau')
\equiv -\left(\frac{b}{p}\right)\sum_{m=0}^{\infty}L_m(\phi_m)
\equiv -\left(\frac{b}{p} \right) \mathbb{D}(F) \pmod p.
$$
This proves the first part of the theorem. Now we prove the second part of the theorem.
Since there exist integers $n$ and $m$ such that $p\nmid nm$ and 
$A_F(n, r, m) \not \equiv 0 \pmod p$, $\mho(F)\not =0$. Therefore 
$\mho(F)=k$ as $p>k$.
Also by the same reason, there exists an integer $m>0$ with $p\nmid m$ 
such that $\phi_m \not \equiv 0 \pmod{p}$
and $\Omega(\phi_m)=k$. Then by \thmref{hjf}, $\Omega(L_m(\phi_m))=k+p+1$. In particular, we have $L_m(\phi_m)\not \equiv 0 \pmod{p}$. Now applying \thmref{enequality}, we deduce that 
$\phi_m$ does not have a Ramanujan-type congruence at $b \pmod{p}$. This implies that $F$ does not have a Ramanujan type congruence at $b\pmod{p}$.
\end{proof}

\section{Examples}

\subsection{$U(p)$ congruences}
We state the following result which will be used to get examples of Hermitian modular forms having $U(p)$ congruences. The proof of the result is obvious.
\begin{prop}\label{Up}
Let $F\in M_k(\Gamma^2(\mathcal{O}), {\rm{det}}^{k/2}, \mathbb{Z}_{(p)})^{sym}$. Then 
$F \mid U(p) \equiv 0 \pmod{p}$ if and only if
$$
\mathbb{D}^{p-1}(F)\equiv F \pmod{p}.
$$
\end{prop}
We consider the Hermitian cusp form 
$\chi_8\in M_8(\Gamma^2(\mathcal{O}), \mbox{det}^4, \mathbb{Z})^{sym}$.
By \thmref{lessthan}, there exists a cusp form 
$H\in M_{32}(\Gamma^2(\mathcal{O}), \mbox{det}^{16}, \mathbb{Z}_{(5)})^{sym}$
such that $\mathbb{D}^{4}(\chi_8) \equiv H \pmod{5}$. Now comparing the coefficients of 
$\mathbb{D}^{4}(\chi_8)$ and $\chi_8$ and using Sturm bound given in \cite[Theorem 2]{nagaoka2},
we deduce that $\mathbb{D}^{4}(\chi_8)\equiv \chi_8 \pmod{5}$. If $p=7$, then by \propref{hermitiancycle}, $\mho(\mathbb{D}(\chi_8))<16$. Thus the possible values of 
$\mho(\mathbb{D}(\chi_8))$ are $4$ and $10$. Since $H_4$ is a non-cusp form,
$\mho(\mathbb{D}(\chi_8))\not =4$. Therefore $\mho(\mathbb{D}(\chi_8))=10$. Now by applying
\propref{hermitiancycle} repeatedly, we deduce that
$\mho(\mathbb{D}^6(\chi_8))=50\not =\mho(\chi_8)=8$. Thus by \propref{Up},
$\chi_8\mid U(7)\not\equiv 0\pmod{7}$. If $p=11$, then  by \thmref{hermitiancong} we deduce that the possible values of $\mho(\mathbb{D}^5(\chi_8))$ are $8$ and $18$. 
If $\mho(\mathbb{D}^5(\chi_8))=8$, then $\mathbb{D}^5(\chi_8)\equiv \beta  \chi_8 \pmod {11}$ for some $\beta \in \{0, 1, \cdots, 10 \}$. We know that 
$A_{\chi_8}(1, (1+i)/2, 1)=1$ and $A_{\chi_8}(1, -1/2, 1)=-486$. Therefore
$\mathbb{D}^5(\chi_8)\not \equiv \beta  \chi_8 \pmod {11}$ for any $\beta \in \{0, 1, \cdots, 10 \}$.
Thus $\mho(\mathbb{D}^5(\chi_8))\not =8$.
Hence $\mho(\mathbb{D}^5(\chi_8))=18$ and $\chi_8\mid U(11)\not \equiv 0$ by \thmref{hermitiancong}.

\subsection{Ramanujan-type congruences}
We use the following result to obtain some examples of Hermitian modular forms having Ramanujan-type congruences. Using \thmref{lessthan} and \thmref{nonzero}, we obtain the following result.
 \begin{thm}\label{belong}
Let $F \in M_k(\Gamma^2(\mathcal{O}), {\rm{det}}^{k/2}, \mathbb{Z}_{(p)})^{sym}$. If 
$$
G:=\mathbb{D}^{\frac{p+1}{2}}(F)+\left(\frac{b}{p}\right)\mathbb{D}(F),
$$
then there exists $H\in M_{k+\frac{(p+1)^2}{2}}(\Gamma^2(\mathcal{O}), 
{\rm{det}}^{\frac{k}{2}+\frac{(p+1)^2}{4}}, \mathbb{Z}_{(p)})^{sym}$ such that
$G\equiv H \pmod{p}$. Moreover, $F$ has a Ramanujan-type congruence at $b\not \equiv 0 \pmod p$ if and only if $G \equiv 0 \pmod p$.
\end{thm}

By \thmref{nonzero}, if 
$F \in M_k(\Gamma^2(\mathcal{O}), {\rm{det}}^{k/2}, \mathbb{Z}_{(p)})^{sym}$ has a 
Ramanujan-type congruence at $b\pmod{p}$, then $p\le k$ or $p=2k-3$. Therefore we use 
 \thmref{belong} and the Sturm bound given in \cite[Theorem 2]{nagaoka2} to get some examples of Hermitian modular forms having Ramanujan-type congruences. The following table consists of examples of Hermitian modular forms of weight $\le 14$ having Ramanujan-type congruences.
 
 \begin{center}
\begin{tabular}{ | c | c | c | c | c | c | } 
\hline
 Hermitian modular forms & $b \pmod p$ \\
 \hline
 $F=\chi_8-6H_4^2$, $F\not\equiv 0 \pmod 7$, $\mathbb{D}(F)\equiv 0 \pmod 7$ & $b \equiv 1, 2, 3, 4, 5, 6 \pmod 7$\\
\hline
$F_{10}$ & $b \equiv 1, 4 \pmod 5$\\
\hline
$H_4F_{10}$ &  $b \equiv 1, 4 \pmod 5 $\\
\hline
$H_4^2H_6+H_6\chi_8$ & $b\equiv 1, 4 \pmod 5 $ \\
\hline

\end{tabular}
\end{center}

\bigskip

\noindent{\bf Acknowledgements.}
We have used the open source mathematics software SAGE to do our computations.
The authors would like to thank  Dr. Soumya Das for  his valuable suggestions.
The research work of the first author was partially supported by the DST-SERB grant MTR/2017/000022.

\end{document}